\documentclass{article}
\usepackage{amsfonts}
\usepackage{eurosym}
\usepackage{amsmath}
\usepackage{amssymb}
\usepackage{latexsym}
\usepackage{color}

\def\qed{ \ \vrule width.2cm height.2cm depth0cm\smallskip}
\newenvironment{proof}{\noindent {\bf Proof.\/}}{$\qed$\vskip 0.1in}

\newcommand{\ba}{\begin{array}}
\newcommand{\ea}{\end{array}}
\newcommand{\be}{\begin{equation}}
\newcommand{\ee}{\end{equation}}
\newcommand{\bea}{\begin{eqnarray}}
\newcommand{\eea}{\end{eqnarray}}
\newcommand{\beaa}{\begin{eqnarray*}}
\newcommand{\eeaa}{\end{eqnarray*}}

\def\dbE{\mathbb{E}}
\def\dbF{\mathbb{F}}

\def\dbP{\mathbb{P}}
\def\dbR{\mathbb{R}}

\def\a{\alpha}

\def\g{\gamma}
\def\d{\delta}
\def\e{\varepsilon}

\def\l{\lambda}

\def\si{\sigma}
\def\t{\tau}
\def\f{\varphi}
\def\th{\theta}
\def\o{\omega}

\def\Th{\Theta}

\def\F{\Phi}
\def\O{\Omega}
\def\cA{{\cal A}}

\def\cD{{\cal D}}

\def\cF{{\cal F}}

\def\cL{{\cal L}}

\def\cP{{\cal P}}

\def\cT{{\cal T}}

\def\no{\noindent}

\def\ms{\medskip}

\def\q{\quad}

\def\pa{\partial}
\def\cd{\cdot}

\def\qed{ \hfill \vrule width.25cm height.25cm depth0cm\smallskip}

\def\limsup{\mathop{\overline{\rm lim}}}
\def\liminf{\mathop{\underline{\rm lim}}}

\def\pa{\partial}

\def\cd{\cdot}

\def\1{{\bf 1}}

\def\:{\!:\!}

at 9pt

\newcommand{\bes}{\begin{subequations}}
\newcommand{\ees}{\end{subequations}}
\newcommand{\bgt}{\begin{gather}}
\newcommand{\egt}{\begin{gather}}

\begin{document}

\title{\textbf{A dual algorithm for stochastic control problems:
Applications to Uncertain Volatility Models and CVA}}
\author{Pierre \textsc{Henry-Labordere}\thanks{%
Soci\'et\'e G\'en\'erale, pierre.henry-labordere@sgcib.com} \and Christian 
\textsc{Litterer}\thanks{%
Centre de Math\'{e}matiques Appliqu\'{e}es, Ecole Polytechnique, Palaiseau,
France, the research of Christian Litterer was supported by ERC grant 321111
RoFiRM.} \and Zhenjie \textsc{Ren}\thanks{%
Centre de Math\'{e}matiques Appliqu\'{e}es, Ecole Polytechnique, Palaiseau,
France, ren@cmap.polytechnique.fr, the research of Zhenjie Ren was supported
by grants from R\'{e}gion Ile-de-France} }
\maketitle

\abstract{
We derive an algorithm in the spirit of Rogers \cite{R} and Davis, Burstein \cite{DB} that leads to upper bounds for stochastic control problems. Our bounds complement
lower biased estimates recently obtained in Guyon, Henry-Labord\`ere \cite{guy}. We evaluate our estimates in
numerical examples motivated from mathematical finance.
}

\newtheorem{thm}{Theorem}[section] \newtheorem{cor}[thm]{Corollary} %
\newtheorem{lem}[thm]{Lemma} \newtheorem{prop}[thm]{Proposition} %
\newtheorem{defn}[thm]{Definition} \newtheorem{rem}[thm]{Remark} %
\newtheorem{exa}[thm]{Example} \newtheorem{assum}[thm]{Assumption}

\renewcommand {\theequation}{\arabic{section}.\arabic{equation}}

\numberwithin{equation}{section} \numberwithin{thm}{section}

\section{Introduction}

\noindent Solving stochastic control problems, for example by approximating
the Hamilton-Jacobi-Bellman (HJB) equation, is an important problem in applied
mathematics. Classical PDE methods are effective tools for solving such
equations in low dimensional settings, but quickly become computationally
intractable as the dimension of the problem increases: a phenomenon commonly
referred to as "the curse of dimensionality". Probabilistic methods on the
other hand such as Monte-Carlo simulation are less sensitive to the
dimension of the problem. It was demonstrated in Pardoux \& Peng \cite{PP}
and Cheridito, Soner, Touzi \& Victoir \cite{CSTV} that first and second
backward stochastic differential equations (in short BSDE) can provide
stochastic representations that may be regarded as a non-linear
generalization of the classical Feynman-Kac formula for semi-linear and
fully non-linear second order parabolic PDEs.

\bigskip

The numerical implementation of such a BSDE based scheme associated to a
stochastic control problem was first proposed in Bouchard \& Touzi \cite%
{BouchardTouziBSDE}, also independently in Zhang \cite{ZhangBSDENum}.
Further generalization was provided in Fahim, Touzi \& Warin \cite{fah} and
in Guyon \& Henry-Labord\`ere \cite{guy}. The algorithm in \cite{guy}
requires evaluating high-dimensional conditional expectations, which are
typically computed using parametric regression techniques. Solving the BSDE
yields a sub-optimal estimation of the stochastic control. Performing an
additional, independent (forward) Monte-Carlo simulation using this
sub-optimal control, one obtains a biased estimation: a lower bound for the
value of the underlying stochastic control problem. Choosing the right basis
for the regression step is in practice a difficult task, particularly in
high-dimensional settings. In fact, a similar situation arises for the
familiar Longstaff-Schwarz algorithm, which also requires the computation of
conditional expectations with parametric regressions and produces a
low-biased estimate.

\bigskip

As the algorithm in \cite{guy} provides a biased estimate, i.e. a lower
bound it is of limited use in practice, unless it can be combined with a
dual method that leads to a corresponding upper bound.
 Such a dual expression was obtained by Rogers \cite{R}, building
on earlier work by Davis and Burstein \cite{DB}. While the work of Rogers is
in the discrete time setting, it applies to a general class of Markov
processes. Previous work by Davis and Burstein \cite{DB} linking
deterministic and stochastic control using flow decomposition techniques
(see also Diehl, Friz, Gassiat \cite{DFG} for a rough path approach to this
problem) is restricted to the control of a diffusion in its drift term. In
the present paper we are also concerned with the control of diffusion
processes, but allow the control to act on both the drift and the volatility
term in the diffusion equation. The basic idea underlying the dual algorithm
in all these works is to replace the stochastic control by a pathwise
deterministic family of control problems that are not necessarily adapted.
The resulting "gain" of information is compensated by introducing a
penalization analogous to a Lagrange multiplier. In contrast to \cite{DB} and
\cite{DFG}, we do not consider continuous pathwise, i.e. deterministic,
optimal control problems. Instead, we rely on a discretization result for
the HJB equation due to Krylov \cite{K-pc} and recover the solution of the
stochastic control problem as the limit of deterministic control problems
over a finite set of discretized controls.

\bigskip

Our paper is structured as follows. In Section \ref{sec:2} we introduce the
stochastic control problem and derive the dual bounds in the Markovian
setting for European type payoffs. In Section \ref{sec:3} we generalize our
estimates to a non-Markovian setting, i.e. where the payoff has a path
dependence. Finally, in Section \ref{sec:4} we consider a setting suitable
for pricing American style options in a Markov setting. We evaluate the
quality of the upper bounds obtained in two numerical examples. First, we
consider the pricing of a variety of options in the uncertain volatility
model. Based on our earlier estimates we transform the stochastic
optimization problem into a family of suitably discretized deterministic
optimizations, which we can in turn approximate for example using local
optimization algorithms. Second, we consider a problem arising in credit
valuation adjustment. In this example, the deterministic optimization can
particularly efficiently be solved by deriving a recursive ODE solution to
the corresponding Hamilton-Jacobi equations. Our algorithm complements the
lower bounds derived in \cite{guy} by effectively re-using some of the
quantities already computed when obtaining the lower bounds (cf. Remark \ref%
{rem:1}).

\section{Duality result for European options\label{sec:2}}

\subsection{Notations}

We begin by introducing some basic notations. For any $k\in \mathbb{N}$ let 
\begin{equation*}
\Omega ^{k}:=\{\omega :\omega \in C([0,T],\mathbb{R}^{k}),\omega _{0}=0\}.
\end{equation*}%
Let $d,m\in \mathbb{N}$ and $T>0$. Define $\Omega :=\Omega ^{d},$ $\Theta
:=[0,T]\times \Omega $ and let $B$ denote the canonical process on $\Omega
^{m}$ with $\mathbb{F}=\{\mathcal{F}_{t}\}_{0\leq t\leq T}$ the filtration
generated by $B$. Finally, denote by $\mathbb{P}_{0}$ the Wiener measure.

For $h>0$, consider a finite partition $\{t_{i}^{h}\}_{i}$ of $[0,T]$ with
mesh less than $h,$ i.e. such that $t_{i+1}^{h}-t_{i}^{h}\leq h$ for all $i$%
. For some $M>0$, let $A$ be a compact subset of 
\begin{equation*}
O_{M}:=\{x\in \mathbb{R}^k:|x|\leq M\},\quad\mbox{for some}~~k\in\mathbb{N},
\end{equation*}%
and $N^{h}$ be a finite $h$-net of $A$, i.e. for all $a,b\in N^h\subset A$,
we have $|a-b|\le h$. We define sets:

\begin{itemize}
\item $\mathcal{A}:=\Big\{\varphi :\Theta \rightarrow \mathbb{R}%
^{k}:~\varphi ~\mbox{is}~\mathbb{F}\mbox{-adapted, and takes values in}~A%
\Big\}$;

\item $\mathcal{A}_h:=\Big\{\varphi\in\mathcal{A}:~ \varphi~%
\mbox{is
constant on}~[t^h_i,t^h_{i+1})~\mbox{for}~i,~\mbox{and takes values in}~N^h%
\Big\}$;

\item $\mathcal{U}:=\Big\{\varphi:\Theta\rightarrow\mathbb{R}^d: ~\varphi~%
\mbox{is bounded and $\dbF$-adapted}\Big\}$;

\item $\mathcal{D}_{h}:=\Big\{f :[0,T]\rightarrow \mathbb{R}^{k}:~f ~%
\mbox{is constant on}~[t_{i}^{h},t_{i+1}^{h})~\mbox{for}~i,~%
\mbox{and
takes values in}~N^{h}\Big\}$.
\end{itemize}

\noindent For the following it is important to note that $\mathcal{D}_{h}$
is a finite set of piecewise constant functions.

We would like to emphasize that, throughout this paper, $C$ denotes a generic constant, which may change from line to line. For example the reader may find $2C\le C$, without any contradiction as the left-hand side $C$ is different from the right-hand side $C$.


\subsection{The Markovian case}

\label{subsec:markov}

We consider stochastic control problems of the form: 
\begin{equation}
u_{0}=\sup_{\alpha \in \mathcal{A}}\mathbb{E}^{\mathbb{P}_{0}}\Big[%
\int_{0}^{T} R^\a_{t} f(t,\alpha_{t},X_{t}^{\alpha })dt + R^\a_{T}g(X_{T}^{\alpha })\Big],  \label{u0}
\end{equation}
where $R^\a_{t}:=e^{-\int_{0}^{t}r(s,\alpha _{s},X_{s}^{\alpha
})d s}$, $X^{\alpha }$ is a $d$-dimensional controlled diffusion defined by%
\begin{equation*}
X^{\alpha }:=\int_{0}^{\cdot }\mu (t,\alpha _{t},X_{t}^{\alpha
})dt+\int_{0}^{\cdot }\sigma (t,\alpha _{t},X_{t}^{\alpha })dB_{t},
\end{equation*}%
and the functions $\mu ,$ $\sigma ,$ $f,$ $r$ satisfy the following
assumption.

\begin{assum}
\label{assum: Markov}The functions $\mu,\sigma ,f,r$ defined on $\dbR^+\times A\times \dbR^d$ take values in $\dbR^d,\dbR^{d\times m},\dbR,\dbR$ respectively. Assume that
\begin{itemize}
\item $\mu,\sigma ,f,r$ are uniformly bounded, and continuous in $\a$;

\item $\mu,\sigma ,f,r$ are uniformly $\d_0$-H\"{o}lder continuous in $t$, for some fixed constant $\d_0\in (0,1]$;

\item $\mu,\si$ are uniformly Lipschitz in $x$, and $f,r$ are uniformly $\d_0$-H\"{o}lder continuous in $x$;

\item $g:\dbR^d\rightarrow \dbR$ is continuous.
\end{itemize}
\end{assum}

\begin{rem}\label{rem:cont-coef}
Our assumptions match the assumptions on the continuity of the coefficients in Krylov \cite{K-pc,K-fd}, and allow us to apply his results.
\end{rem}

\bigskip Our main result is a duality in the spirit of \cite{DB} that allows
us to replace the stochastic control problem by a family of suitably
discretized deterministic control problems.  We first discretize the control problem through the following lemma which is a direct consequence of Theorem
2.3 in Krylov \cite{K-pc}.

\no Define the function 
\begin{equation*}
u_{0}^{h}:=\sup_{\alpha \in \mathcal{A}_{h}}\mathbb{E}^{\mathbb{P}_{0}}\Big[%
\int_{0}^{T}R^\a_{t}f(t,\alpha_{t},X_{t}^{\alpha })dt + R^\a_{T} g(X_{T}^{\alpha })\Big].
\end{equation*}
\begin{lem}
\label{lem:krylov}Suppose Assumption \ref{assum: Markov} holds and $g$ is
bounded. We have for any family of partition of $[0,T]$ with mesh tending to zero that
\begin{equation}
u_{0}=\lim_{h\rightarrow 0}u_{0}^{h}.  \label{n-pbm}
\end{equation}
\end{lem}

\begin{rem}
\bigskip \label{rem: rate}Theorem 2.3 in \cite{K-pc} also gives a rate of
convergence for the discretization in Lemma \ref{lem:krylov}, i.e. there exists
a constant $C>0$ such that%
\begin{equation*}
\left\vert u_{0}-u_{0}^{h}\right\vert \leq Ch^{\frac{1}{3}},
\q \mbox{for all}~~0<h\leq 1.
\end{equation*}
\end{rem}

For the following statement, we introduce: 
\be\label{equation defining approximates}
\left.\ba{lll}
& v^{h}:=\inf_{\varphi \in \mathcal{U}}\mathbb{E}^{\mathbb{P}_{0}}\Big[\max_{a
\in \mathcal{D}_{h}}\F^{a,\f} \Big],~~\mbox{with}&\\
&\F^{a,\f} := R^a_{T} g(X_{T}^{a })+\int_{0}^{T}R^a_{t}f(t,a _{t},X_{t}^{a })dt 
-\int_{0}^{T} R^a_{t} \varphi _{t}(X^{a })^{\intercal}\sigma (t,a _{t},X_{t}^{a })dB_{t}. &
\ea\right.
\ee
\begin{rem}{\rm
It is noteworthy that stochastic integrals are defined in $L^2$-space, so it is in general meaningless to take the pathwise supremum of a family of stochastic integrals. However, as we mentioned before, the set $\cD_h$ is of finite elements. So there is a unique random variable in $L^2$ equal to the maximum value of the finite number of stochastic integrals, $\dbP_0$-a.s.
}
\end{rem}

The next theorem allows to recover the stochastic optimal control
problem as a limit of discretized deterministic control problems.

\begin{thm}
\label{thm: markov} Suppose Assumption \ref{assum: Markov} holds and $g$ is
bounded. Then we have%
\begin{equation*}
u_{0}=\lim_{h\rightarrow 0}v^{h}.
\end{equation*}
\end{thm}

\begin{proof}
We first prove that $u_{0}\leq \liminf_{h\rightarrow 0}v^{h}.$ Recall $%
u_{0}^{h} $ defined in \eqref{n-pbm}. Since $R^\a, \si$ are bounded, for all $\varphi \in \mathcal{U}$ the
process $\int_{0}^{\cdot } R^\a_t \varphi _{t}(X^{\alpha })^{\intercal}\sigma (t,\alpha
_{t},X_{t}^{\alpha })dB_{t}$ is a martingale. So we have 
\begin{eqnarray*}
u_{0}^{h} &= &\sup_{\alpha \in \mathcal{A}_{h}}\mathbb{E}^{\mathbb{P}_{0}}%
\big[\F^{\a,\f} \big].
\end{eqnarray*}
Since $\F^{\a,\f} \le \max_{a\in \cD_h} \F^{a,\f}$ for all $\a\in \cA_h$, we have
\beaa
u_{0}^{h} &\le & \dbE^{\dbP_0}\big[\max_{a\in \cD_h} \F^{a,\f}\big].
\eeaa
The required result follows.

\medskip \noindent To show $u_{0}\geq \limsup_{h\rightarrow 0}v^{h}$ we
construct an explicit minimizer $\varphi ^{\ast }$. First note that under
Assumption \ref{assum: Markov}, it is easy to verify that $u_{t}$ defined as 
\begin{equation*}
u(t,x):=\sup_{\alpha \in \mathcal{A}}\mathbb{E}^{\mathbb{P}_{0}}\Big[%
\int_{t}^{T} \frac{R^\a_s}{R^\a_t}f(s,\alpha _{s},X_{s}^{\alpha })ds+\frac{R^\a_T}{R^\a_t} g(X_{T}^{\alpha })\Big|X_{t}^{\alpha }=x\Big],
\end{equation*}%
is a viscosity solution to the Dirichlet problem of the HJB equation: 
\begin{equation}  \label{HJB eq}
\left.%
\begin{array}{lll}
& -\partial _{t}u-\sup_{b \in A}\big\{ \mathcal{L}^b u + f(t,b,x) \big\} =0,\quad u_{T}=g, &  \\ 
& \mbox{where}~~\mathcal{L}^b u := \mu (t,b ,x)\cdot \partial _{x}u+\frac{1}{%
2}\mathrm{Tr}\big((\sigma \sigma ^{\intercal})(t,b ,x)\partial _{xx}^{2}u%
\big)-r(t,b ,x)u. & 
\end{array}%
\right.
\end{equation}

We next define the mollification $u^{(\varepsilon )}:=u\ast K^{(\varepsilon
)}$ of $u$, where $K$ is a smooth function with compact support in $%
(-1,0)\times O_{1}$ ($O_{1}$ is the unit ball in $\mathbb{R}^{d}$), and $%
K^{(\varepsilon )}(x):=\varepsilon ^{-n-2}K(t/\varepsilon ^{2},x/\varepsilon
)$. Clearly, $u^{(\varepsilon )}\in C_{b}^{\infty }$ and $u^{(\varepsilon )}$
converges uniformly to $u$. 
As mentioned in Remark \ref{rem:cont-coef}, Assumption \ref{assum: Markov} matches the assumptions in \cite{K-fd}, where the author proved in his Theorem 2.1 that 
\begin{center}
$u^{(\varepsilon )}$ is a
classical supersolution to the HJB equation \eqref{HJB eq}.
\end{center}
Denote 
\be\label{choice_phi}
\f^\e_t(\o) := \pa_x u^{(\e)}(t,\o_t).
\ee
Since $u^{(\varepsilon )}\in C_{b}^{\infty }$, it follows from the It\^o's formula that
\begin{multline*}
R^a_T u^{(\varepsilon )}(T, X_{T}^{a
})-u_{0}^{(\varepsilon )} ~=~\int_{0}^{T} R^a_t \big( \pa_t u^{(\e)}+\mathcal{L}^{a_t} u^{(\varepsilon)}(t,X^a_t)\big)dt \\
+\int_{0}^{T} R^a_t \f^\e_t(X^{a })^{\intercal}\sigma (t,a _{t},X_{t}^{a
})dB_{t},\quad \mbox{for all}~a \in \mathcal{D}_{h},\quad \mathbb{P}_{0}%
\mbox{-a.s.}
\end{multline*}
Then, by the definition of $\F^{a,\f^\e}$ in \eqref{equation defining approximates}, we obtain
\begin{multline*}
\F^{a,\f^\e} = R^a_T g(X^a_T) + \int_0^T R^a_t \Big( f(t,a_t,X^a_t)+  \big(\pa_t u^{(\e)} + \cL^{a_t} u^{(\e)}\big)(t,X^a_t)\Big) dt\\
 - R^a_T u^{(\varepsilon )}(T, X_{T}^{a})+u_{0}^{(\varepsilon )},\quad \mbox{for all}~a \in \mathcal{D}_{h},\quad \mathbb{P}_{0}\mbox{-a.s.}
\end{multline*}
Since $u^{(\varepsilon )}$ is a supersolution to the HJB equation \eqref{HJB eq}, it follows that
\be\label{estime:Phi}
\F^{a,\f^\e}  ~\le ~ R^a_T \big(%
g(X_{T}^{\alpha })-u^{(\varepsilon )}(T, X_{T}^{\alpha })\big)%
+u_{0}^{(\varepsilon )},\quad \mbox{for all}~a \in \mathcal{D}_{h},\quad \mathbb{P}_{0}%
\mbox{-a.s.}
\ee
By Assumption \ref{assum: Markov} and the fact that $g$ is bounded, 
\be\label{eq:Phibounded}
\mbox{$\F^{a,\f^\e}$
is uniformly bounded from above.}
\ee
 Also, it is easy to verify that the function $u$
is continuous and therefore uniformly continuous on $S_{L}:=[0,T]\times
\{|x|\leq L\}$ for any $L>0$ and that $u^{(\varepsilon )}$ converges uniformly to 
$u$ on $S_{L}$. In particular, 
\be\label{eq:conv_ueps}
\left.\ba{lll}
& u^{(\e)}_0\rightarrow u_0, &\\
&\rho _{L}\left( \varepsilon \right) :=\max_{\left\vert x\right\vert \leq
L}\left\vert g\left( x\right) -u^{\left( \varepsilon \right) }\left(T,
x\right) \right\vert  \rightarrow 0,&
\ea\right.
\q \mbox{as}~~ \e\rightarrow 0.
\ee
It follows from \eqref{estime:Phi}, \eqref{eq:Phibounded} and \eqref{eq:conv_ueps} that
\begin{eqnarray*}
\mathbb{E}^{\mathbb{P}_{0}}\Big[\max_{a \in 
\mathcal{D}_{h}} \F^{a, \f^\e}\Big] 
&=& \mathbb{E}^{\mathbb{P}_{0}}\Big[\max_{a \in \mathcal{D}_{h}} \F^{a,\f^\e };
\max_{a \in \mathcal{D}_{h}}|X_{T}^{a }|\leq L\Big]+\mathbb{E}^{\mathbb{P}_{0}}\Big[\max_{a \in 
\mathcal{D}_{h}} \F^{a,\f^\e };\max_{a \in \mathcal{D}_{h}}|X_{T}^{a
}|>L\Big] \\
&\leq & C\rho _{L}(\varepsilon )+u_{0}^{(\varepsilon )}+C\mathbb{P}_{0}\big[%
\max_{a \in \mathcal{D}_{h}}|X_{T}^{a }|>L\big],
\end{eqnarray*}%
where $C$ is a constant independent of $L$ and $\varepsilon$. Therefore 
\begin{equation*}
v^h ~\le ~ \liminf_{\e\rightarrow 0} \mathbb{E}^{\mathbb{P}_{0}}\Big[\max_{a \in 
\mathcal{D}_{h}} \F^{a, \f^\e}\Big] 
~\leq ~ u_{0}+C\mathbb{P}_{0}\big[\max_{a \in \mathcal{D}_{h}}|X_{T}^{a
}|>L\big], \q\mbox{for any}~~L>0.
\end{equation*}%
 Further, since 
\beaa
\mathbb{P}_{0}\big[\max_{a \in \mathcal{D}_{h}}|X_{T}^{a }|>L\big] 
~\le ~ \sum_{a\in \cD_h} \dbP_0\big[ |X_{T}^{a }|>L\big] ~\rightarrow ~ 0, \q \mbox{as} ~~L\rightarrow \infty,
\eeaa
 we conclude that $v^h \le u_0$. So the required inequality follows.
\end{proof}

The boundedness assumption on $g$ may be relaxed by means of a simple cut
off argument:

\begin{cor}
\label{cut off} Assume that $g$ is of polynomial growth, i.e.
$$|g(x)| \le C\big(1+|x|^p\big), \q\mbox{for some}~~ C, p \ge 0.$$
Let $M>0$ ,$g^{M}$\ a continuous compactly supported function
that agrees with $g$ on $O_{M}\subseteq \mathbb{R}^{d}$ and satisfies $%
\left\vert g^{M}\right\vert \leq \left\vert g\right\vert $. Let $v^{h,M}$
denote the approximations defined in $\left( \ref{equation defining
approximates}\right) ,$ with respect to $g^{M}$ in place of $g.$ Then we have
\begin{equation*}
\lim_{M\rightarrow 0}\left\vert u_{0}-\lim_{h\rightarrow 0}v^{h,M}\right\vert ~=~ 0.
\end{equation*}
\end{cor}

\begin{proof}
Define $u_0^M$ as in \eqref{u0} by using the approximation $g^M$, i.e.
\beaa
u^M_{0} &:=&\sup_{\alpha \in \mathcal{A}}\mathbb{E}^{\mathbb{P}_{0}}\Big[%
\int_{0}^{T} R^\a_{t}
f(t,\alpha_{t},X_{t}^{\alpha })dt + R^\a_{T} g^M(X_{T}^{\alpha })\Big].
\eeaa
 By Theorem %
\ref{thm: markov}, we know that $u_0^M=\lim_{h\rightarrow 0} v^{h,M}$.

Further, we have 
\begin{eqnarray*}
|u_0-u_0^M| &\le & C\sup_{\alpha\in \mathcal{A}}\mathbb{E}^{\mathbb{P}_0}%
\Big[g(X^\alpha_T)-g^M(X^\alpha_T)\Big] \\
&\le & C\sup_{\alpha\in\mathcal{A}}\mathbb{E}^{\mathbb{P}_0}\Big[%
|X^\alpha_T|^p+1; |X^\alpha_T|\ge M\Big].
\end{eqnarray*}
Assume $M\ge 1$. Then we obtain
\be\label{diffuM}
|u_0-u_0^M| ~\le ~ C\sup_{\alpha\in\mathcal{A}}\mathbb{E}^{\mathbb{P}_0}\Big[%
|X^\alpha_T|^p; |X^\alpha_T|\ge M\Big]
~\le ~ C \sup_{\alpha\in\mathcal{A}}\mathbb{E}^{\mathbb{P}_0}\Big[%
\frac{|X^\alpha_T|^{p+1}}{M} \Big].
\ee
Since $\mu,\si$ are both bounded, we have
\be\label{estimate:momentX}
\dbE^{\dbP_0}\Big[|X^\alpha_T|^{p+1}\Big] 
~ \le ~ C \dbE^{\dbP_0}\Big[ \big|\int_0^T \mu(t,\a_t, X^\a_t)dt \big|^{p+1}
	+ \big|\int_0^T \si(t,\a_t, X^\a_t)d B_t \big|^{p+1} \Big]
~ \le ~ CT.
\ee
It follows from \eqref{diffuM} and \eqref{estimate:momentX} that
\beaa
\lim_{M\rightarrow\infty} |u_0-u_0^M| &=& 0.
\eeaa
 The proof is completed.
\end{proof}

We conclude the section with two remarks, both relevant to the numerical
simulation of the approximation derived in Theorem \ref{thm: markov}.

\begin{rem}\label{rem:1}
To approximate $v^{h}$ in our numerical examples we will as in
the proof of Theorem \ref{thm: markov} use fixed functions $\varphi
^{\ast }$ for the minimization. The definition \eqref{choice_phi} makes it clear that the natural
choice for these minimizers are (the numerical approximations of) the
function $\partial _{x}u$. Note that these approximations are readily
available from the numerical scheme in \cite{guy} that is used to compute the
complementary lower bounds.
\end{rem}

\begin{rem}
In the proof of Theorem \ref{thm: markov} we showed that $u_{0}^{h}\leq
v^{h}\leq u_{0}.$ It therefore follows from Remark \ref{rem: rate} that
there exists a constant $C>0$ such that%
\begin{equation*}
\left\vert u_{0}-v^{h}\right\vert \leq Ch^{\frac{1}{3}}
\end{equation*}%
for all $0<h\leq 1\wedge T.$
\end{rem}

\section{Some extensions}

\subsection{The non-Markovian case\label{sec:3}}

In our first extension we consider stochastic control problems of the form 
\begin{equation*}
u_{0}=\sup_{\alpha \in \mathcal{A}}\mathbb{E}^{\mathbb{P}_{0}}\Big[%
g(X_{T\wedge \cdot }^{\alpha })\Big],
\end{equation*}%
where $X^{\alpha }$ is a $d-$dimensional diffusion defined by $X^{\alpha
}:=\int_{0}^{\cdot }\mu (t,\alpha _{t})dt+\int_{0}^{\cdot }\sigma (t,\alpha
_{t})dB_{t}$. Note that in this setting $\mu $ and $\sigma $ only depend on $%
\alpha $ and $t,$ but the payoff function $g$ is path dependent.

\begin{rem}{\rm
The arguments in this subsection are based on the "frozen-path"
approach developed in Ekren, Touzi and Zhang \cite{ETZ}. In order to apply their approach, we have restricted the class of diffusions $X^{\alpha }$ we consider, compared to the Markovian control problem.
}
\end{rem}

Writing $\mathbb{P}_{\alpha }:=\mathbb{P}_{0}\circ (X^{\alpha })^{-1},$ we
have 
\begin{equation*}
u_{0}=\sup_{\alpha \in \mathcal{A}}\mathbb{E}^{\mathbb{P}_{\alpha }}\Big[%
g(B_{T\wedge \cdot })\Big].
\end{equation*}%
Throughout this subsection we will impose the following regularity
assumptions.

\begin{assum}\label{assum:non Markov}
The functions $\mu,\si: \dbR^+\times A \rightarrow E$ ($E$ is the respective metric space) and $g:\O^d \rightarrow \dbR$ are uniformly bounded such that
\begin{itemize}
\item $\mu,\si$ are continuous in $\a$;

\item $\mu,\si$ are $\d_0$-H\"older continuous in $t$, for some constant $\d_0\in (0,1]$;

\item $g$ is uniformly continuous.  
\end{itemize}
\end{assum}\label{assum: non-Markov}

\begin{exa}
Arguing as in Corollary \ref{cut off} we may also consider unbounded
payoffs. Hence, possible path-dependent payoffs that fit our framework
include e.g. the maximum $\max_{s\in \left[ 0,T\right] }\omega _{s}$ and
Asian options $\frac{1}{T}\int_{0}^{T}\omega _{s}ds.$
\end{exa}

Let 
$$\Lambda_{\varepsilon }:=\big\{t_{0}=0,t_{1},t_{2},\cdots ,t_{n}=T\big\}$$
be a partition of $[0,T]$ with mesh bounded above by $\varepsilon $. For $%
k\leq n$ and $\pi _{k}=(x_{1}=0,x_{2},\cdots ,x_{k})\in \mathbb{R}^{d\times
k}$, denote by $\Gamma _{\varepsilon }^{\Lambda_{\varepsilon },k}(\pi _{k})$
the path generated by the linear interpolation of the points $%
\{(t_{i},x_{i})\}_{0\leq i\leq k}$. Where no confusion arises with regards
to the underlying partition we will in the following drop the superscript $%
\Lambda_{\varepsilon }$ and write $\Gamma _{\varepsilon }^{k}(\pi _{k})$ in
place of $\Gamma _{\varepsilon }^{\Lambda_{\varepsilon },k}(\pi _{k}),$ but
it must be emphasized that the entire analysis in this subsection is carried
out with a fixed but arbitrary partition $\Lambda_{\varepsilon }$ in mind.
Define the interpolation approximation of $g$ by 
\begin{equation*}
g^{\varepsilon }(\pi _{n}):=g\Big(\Gamma _{\varepsilon }^{n}(\pi _{n})\Big)
\end{equation*}%
and define an approximation of the value function by letting 
\begin{equation*}
\theta _{0}^{\varepsilon }:=\sup_{\alpha \in \mathcal{A}}\mathbb{E}^{\mathbb{%
P}_{\alpha }}\Big[g^{\varepsilon }\big((B_{t_{i}})_{0\leq i\leq n}\big)\Big].
\end{equation*}%
The following lemma justifies the use of linear interpolation for
approximating dependent payoff.

\begin{lem}
\label{lem: th eps u} Under Assumption \ref{assum: non-Markov}, we have
\begin{equation*}
\lim_{\varepsilon \rightarrow 0}\theta _{0}^{\varepsilon }=u_{0}.
\end{equation*}
\end{lem}

\begin{proof}
Recall that $g$ is uniformly continuous. Let $\rho $ be a modulus of
continuity of $g$. If necessary, we may choose $\rho $ to be concave (by taking the concave envelop).
Further, we define 
\begin{equation*}
\mathbf{w}_{B}(\varepsilon ,T):=\sup_{s,t\leq T;|s-t|\leq \varepsilon
}|B_{s}-B_{t}|.
\end{equation*}%
Clearly, we have 
\beaa
|\th^\e_0 - u_0| &=&\Big|\sup_{\alpha \in \mathcal{A}}\mathbb{E}^{\mathbb{P}_{\alpha }}\big[%
g^{\varepsilon }\big((B_{t_{i}})_{0\leq i\leq n}\big)\big]-\sup_{\alpha \in 
\mathcal{A}}\mathbb{E}^{\mathbb{P}_{\alpha }}\big[g(B_{T\wedge \cdot })\big]%
\Big|\\
&\leq & \sup_{\alpha \in \mathcal{A}}\mathbb{E}^{\mathbb{P}_{\alpha }}\big[%
\rho \big(\mathbf{w}_{B}(\varepsilon ,T)\big)\big]
~\leq ~ \rho \Big(%
\sup_{\alpha \in \mathcal{A}}\mathbb{E}^{\mathbb{P}_{\alpha }}\big[\mathbf{w}%
_{B}(\varepsilon ,T)\big]\Big).
\eeaa 
It is proved in Theorem 1 in Fisher and Nappo \cite{FN} that 
\beaa
\mathbb{E}^{\mathbb{P}_{\alpha }}\big[\mathbf{w}_{B}(\varepsilon ,T)\big]
~\le ~ C \Big(\e \ln \frac{2T}{\e}\Big)^\frac12,
\eeaa
where $C$ is a constant only dependent on the bound of $\mu$ and $\si$.
Thus,
\beaa
\lim_{\e\rightarrow 0}\sup_{\alpha \in \mathcal{A}}\mathbb{E}^{\mathbb{P}_{\alpha }}\big[\mathbf{w}_{B}(\varepsilon ,T)\big] &=& 0. 
\eeaa
The proof is completed.
\end{proof}

We next define the controlled diffusion with time-shifted coefficients by
setting 
\begin{equation*}
X^{\alpha ,t}:=\int_{0}^{s}\mu (t+r,\alpha _{r})dr+\int_{0}^{s}\sigma
(t+r,\alpha _{r})dB_{r},~s\in \lbrack 0,T-t],~\mathbb{P}_{0}\mbox{-a.s.},
\end{equation*}%
and the corresponding law: 
\begin{equation*}
\mathbb{P}_{\alpha }^{t}:=\mathbb{P}_{0}\circ (X^{\alpha ,t})^{-1}.
\end{equation*}%
Further,\ for $1\leq k\leq n-2$ let 
$$\eta _{k}:=t_{k+1}-t_{k},$$
 and define
recursively a family of stochastic control problems: 
\bea \label{def:th_eps}
& \theta ^{\varepsilon }(\pi _{n-1};t,x):=\sup_{\alpha \in \mathcal{A}}%
\mathbb{E}^{\mathbb{P}_{\alpha }^{t_{n-1}+t}}\Big[g^{\varepsilon }\big((\pi
_{n-1},x_{n-1}+x+B_{\eta _{n-1}-t})\big)\Big],\quad t\in \lbrack 0,\eta
_{n-1}),~x\in \mathbb{R}^{d} & \notag \\ 
& \theta ^{\varepsilon }(\pi _{k};t,x):=\sup_{\alpha \in \mathcal{A}}\mathbb{%
E}^{\mathbb{P}_{\alpha }^{t_{k}+t}}\Big[\theta ^{\varepsilon }\big((\pi
_{k},x_{k}+x+B_{\eta _{k}-t}),0,0\big)\Big],\quad t\in \lbrack 0,\eta
_{k}),~x\in \mathbb{R}^{d}. & 
\eea
Clearly, $\theta ^{\varepsilon }(0;0,0)=\theta _{0}^{\varepsilon }$.
\begin{rem}{\rm
By freezing the path $\pi_k$, we get the value function $\th^\e(\pi_k;\cd,\cd)$ of a {\it Markovian} stochastic control problem on the small interval $[0,\eta_k)$. This will allow us to apply the PDE tools which played a key role in proving the dual form in the previous section.
}
\end{rem}

\begin{lem}
\label{lem: th^h continuous} Fix $\varepsilon >0.$ The function $\theta
^{\varepsilon }(\pi ;t,x)$ is Borel-measurable in all the arguments and
uniformly continuous in $(t,x)$ uniformly in $\pi $.
\end{lem}

\begin{proof}
It follows from the uniform continuity of $g$ and the fact that
interpolation with respect to a partition $\Lambda_{\varepsilon }$ is a
Lipschitz function (in this case from $\mathbb{R}^{n\times d}$ into the
continuous functions), that $g^{\varepsilon }$ is also uniformly continuous.
Denote by $\rho ^{\varepsilon }$ a modulus of continuity of $g^{\varepsilon
} $, chosen to be increasing and concave if necessary. For any $\pi
_{n-1},\pi _{n-1}^{\prime }\in \mathbb{R}^{\left( n-1\right) \times d}$,
given $t\in \lbrack 0,\eta _{n-1}]$, $x,x^{\prime }\in \mathbb{R}^{d}$, we
have 
\begin{eqnarray*}
&&|\theta ^{\varepsilon }(\pi _{n-1};t,x)-\theta ^{\varepsilon }(\pi
_{n-1}^{\prime };t,x^{\prime })| \\
&\leq &\sup_{\alpha \in \mathcal{A}}\mathbb{E}^{\mathbb{P}_{\alpha
}^{t_{n-1}+t}}\Big[\Big|g^{\varepsilon }\big((\pi _{n-1},x_{n-1}+x+B_{\eta
_{n-1}-t})\big)-g^{\varepsilon }\big((\pi _{n-1}^{\prime },x_{n-1}+x^{\prime
}+B_{\eta _{n-1}-t})\big)\Big|\Big] \\
&\leq &\rho ^{\varepsilon }(|(\pi _{n-1},x)-(\pi _{n-1}^{\prime },x^{\prime
})|).
\end{eqnarray*}%
Similarly, for any $k<n-1$ and $\pi _{k},\pi _{k}^{\prime }\in \mathbb{R}%
^{k\times d}$, given $t\in \lbrack 0,\eta _{k}]$, $x,x^{\prime }\in \mathbb{R%
}^{d}$, we have 
\bea \label{eq: diff x}
&  & |\theta ^{\varepsilon }(\pi _{k};t,x)-\theta ^{\varepsilon }(\pi
_{k}^{\prime };t,x^{\prime })| \notag \\ 
& \leq & \sup_{\alpha \in \mathcal{A}}\mathbb{E}^{\mathbb{P}_{\alpha
}^{t_{k}+t}}\Big[\Big|\theta ^{\varepsilon }\big((\pi _{k},x_{k}+x+B_{\eta
_{k}-t}),0,0\big)-\theta ^{\varepsilon }\big((\pi _{k}^{\prime
},x_{k}+x^{\prime }+B_{\eta _{k}-t}),0,0\big)\Big|\Big] \notag\\ 
& \leq & \rho ^{\varepsilon }(|(\pi _{k},x)-(\pi _{k}^{\prime },x^{\prime
})|).%
\eea
For $0\leq t^{0}<t^{1}\leq \eta _{k}$, it follows from the dynamic
programming principle (for a general theory on the dynamic programming principle for sublinear expectations, we refer to Nutz and Van Handel \cite{NvH}) that 
\begin{equation}
\theta ^{\varepsilon }(\pi _{k};t^{0},x)=\sup_{\alpha \in \mathcal{A}}%
\mathbb{E}^{\mathbb{P}_{\alpha }^{t_{k}+t^{0}}}\Big[\theta ^{\varepsilon
}(\pi _{k};t^{1},x+B_{t^{1}-t^{0}}))\Big]  \label{dpp}
\end{equation}%
and \eqref{dpp} and \eqref{eq: diff x} we deduce that 
\bea \label{eq: diff t 1}
|\theta ^{\varepsilon }(\pi _{k};t^{0},x)-\theta ^{\varepsilon }(\pi
_{k};t^{1},x)| & \leq & \sup_{\alpha \in \mathcal{A}}\mathbb{E}^{\mathbb{P}%
_{\alpha }^{t_{k}+t^{0}}}\Big[\Big|\theta ^{\varepsilon }(\pi
_{k};t^{1},x+B_{t^{1}-t^{0}}))-\theta ^{\varepsilon }(\pi _{k};t^{1},x)\Big|%
\Big] \notag \\ 
& \leq & \sup_{\alpha \in \mathcal{A}}\mathbb{E}^{\mathbb{P}_{\alpha
}^{t_{k}+t^{0}}}\big[\rho ^{\varepsilon }(|B_{t^{1}-t^{0}}|)\big] \notag \\ 
& \leq & \rho ^{\varepsilon }\Big(\sup_{\alpha \in \mathcal{A}}\mathbb{E}^{%
\mathbb{P}_{\alpha }^{t_{k}+t^{0}}}\big[|B_{t^{1}-t^{0}}|\big]\Big).
\eea
Similar to \eqref{estimate:momentX}, we have the estimate:
\begin{equation}\label{eq: diff t 2}
\sup_{\alpha \in \mathcal{A}}\mathbb{E}^{\mathbb{P}_{\alpha }^{t_{k}+t^{0}}}%
\Big[|B_{t^{1}-t^{0}}|\Big]
~=~ \sup_{\alpha \in \mathcal{A}}\mathbb{E}^{\mathbb{P}_0}\Big[|X^{\a,t_{k}+t^{0}}_{t^{1}-t^{0}}|\Big]
~\leq ~ C \big(t^1-t^0 \big),
\end{equation}%
where $C$ is a constant only dependent on the bound of $\mu$ and $\si$. It follows from 
\eqref{eq: diff t 1} and \eqref{eq: diff t 2} that 
\beaa
|\theta ^{\varepsilon }(\pi _{k};t^{0},x)-\theta ^{\varepsilon }(\pi
_{k};t^{1},x)|~\leq ~\rho ^{\varepsilon }\Big(C(t^1-t^0)\Big)~.
\eeaa
Hence, combining \eqref{eq: diff x} and \eqref{eq: diff t 2} we conclude
that $\theta ^{\varepsilon }(\pi _{k};t,x)$ is uniformly continuous in $%
(t,x) $ uniformly in $\pi _{k}$.
\end{proof}

The functions $\theta ^{\varepsilon }(\pi _{k};\cdot ,\cdot )$ are defined
as the value functions of stochastic control problems, and one can easily
check that they are viscosity solutions to the corresponding
Hamilton-Jacobi-Bellman equations. For $k=1,\ldots ,n-1,$ we define a family
of PDEs by letting 
\begin{equation}
\left. 
\begin{array}{lll}
& \quad \quad -\mathbf{L}^{k}\theta ~=~ 0,\quad \mbox{on}~~[0,\eta _{k})\otimes 
\mathbb{R}^{d}, ~\mbox{where}&  \\ 
& \mathbf{L}^{k}\theta ~:=~ \partial _{t}\theta + \sup_{b \in
A}\Big\{\mu \big(t_{k}+\cdot , b \big)\cdot \partial _{x}\theta +\frac{1}{2}\mathrm{Tr}\big((\sigma \sigma ^{\intercal})(t_{k}+\cdot , b
)\partial _{xx}^{2}\theta \big)\Big\}. & 
\end{array}%
\right.  \label{pathPDE}
\end{equation}%
The following proposition links the stochastic control problems with the PDE
and applies, analogous to the Markovian case, a mollification argument.

\begin{prop}
\label{prop u eps} There exists a function $u^{(\varepsilon )}:(\pi
,t,x)\mapsto \mathbb{R}$ such that $u^{(\varepsilon )}(0,0,0)=\theta
_{0}^{\varepsilon }+\varepsilon $ and for all $\pi _{k}$, $u^{(\varepsilon )}(\pi _{k};\cdot ,\cdot )$ is a classical
supersolution to the PDE $\left( \ref{pathPDE}\right) $ and the
boundary condition: 
\begin{equation*}
\left. 
\begin{array}{lll}
u^{(\varepsilon )}(\pi _{k};\eta _{k},x)=u^{(\varepsilon )}\big((\pi
_{k},x);0,0\big),\quad \quad \mbox{if}~k<n-1; &  &  \\ 
u^{(\varepsilon )}(\pi _{k};\eta _{k},x)\geq g^{\varepsilon }\big((\pi
_{k},x)\big),\quad \quad \quad \quad \mbox{if}~k=n-1. &  & 
\end{array}%
\right.
\end{equation*}
\end{prop}

\begin{proof}
Define $\theta ^{\varepsilon ,\delta }(\pi _{k};\cdot ,\cdot ):=\theta
^{\varepsilon }(\pi _{k};\cdot ,\cdot )\ast K^{\delta }$ for all $\pi
_{k}\in \mathbb{R}^{k\times d}$, $k\leq n$,where $K$ is a smooth function
with compact support in $(-1,0)\times O_{1}$ ($O_{1}$ is the unit ball in $%
\mathbb{R}^{d}$), and $K^{\delta }(t,x):=\delta ^{-d-2}K(t/\delta
^{2},x/\delta )$. By Lemma \ref{lem: th^h continuous}, $\theta ^{\varepsilon
,\delta }(\pi _{k};\cdot ,\cdot )$ converges uniformly to $\theta
^{\varepsilon }(\pi _{k};\cdot ,\cdot )$ uniformly in $\pi _{k}$, as $\delta
\rightarrow 0$. Take $\delta $ small enough so that $\Vert \theta
^{\varepsilon ,\delta }-\theta ^{\varepsilon }\Vert \leq \frac{\varepsilon }{%
2n}$. Further, Assumption \ref{assum:non Markov} implies that all the shifted coefficients $\mu(t_k+\cd,\cd), \si(t_k+\cd,\cd)$ satisfy the assumptions on the continuity of the coefficients in \cite{K-fd}, where the author 
proved that 
\begin{center}
$\theta ^{\varepsilon ,\delta }(\pi _{k};\cdot ,\cdot )$ is a
classical supersolution for $\left( \ref{pathPDE}\right) $.
\end{center}
 Note that $%
\theta^{\varepsilon ,\delta }(\pi _{k};\cdot ,\cdot )+C$ is still a
supersolution for any constant $C$. So we may define a smooth function $%
v^{\varepsilon }(0;\cdot ,\cdot) :=  \theta^{\varepsilon ,\delta }(0;\cdot ,\cdot )+C_0$ on $[0,t_{1}]\times \mathbb{R}^{d}$ with some constant $C_0$ such that
\begin{equation*}
v^{\varepsilon }(0;0,0)=\theta ^{\varepsilon }(0;0,0)+\frac{\varepsilon }{n}%
,\quad v^{\varepsilon }(0;\cdot ,\cdot )\geq \theta ^{\varepsilon }(0;\cdot
,\cdot ).
\end{equation*}%
Similarly, we define smooth functions $v^{\varepsilon }(\pi _{k};\cdot ,\cdot) := \theta^{\varepsilon ,\delta }(\pi _{k};\cdot ,\cdot )+C_{\pi_k}$ on $[0,\eta
_{k}]\times \mathbb{R}^{d}$ for $1\leq k\leq n-1$  with some constants $C_{\pi_k}$ such that 
\begin{equation*}
v^{\varepsilon }(\pi _{k};0,0)=v^{\varepsilon }(\pi _{k-1};\eta
_{k-1},x_{k}-x_{k-1})+\frac{\varepsilon }{n},\q
v^{\varepsilon }(\pi
_{k};\cdot ,\cdot )\geq \theta ^{\varepsilon }(\pi _{k};\cdot ,\cdot )~.
\end{equation*}%
Finally, we define for $\pi _{k}\in \mathbb{R}^{k\times d}$ and $(t,x)\in
\lbrack 0,\eta _{k})\times \mathbb{R}^{d}$ 
\begin{equation*}
u^{(\varepsilon )}(\pi _{k};t,x):=v^{\varepsilon }(\pi _{k};t,x)+\frac{n-k+1%
}{n}\varepsilon .
\end{equation*}%
It is now straightfoward to check that $u^{(\varepsilon )}$ satisfies the
requirements.
\end{proof}

The discrete framework we just developed may be linked to pathspace by means
of linear interpolation along the partition $\Lambda_{\varepsilon }$. Recall
that $\Theta $ was defined to be $\left[ 0,T\right] \times \Omega .$

\begin{cor}
\label{cor: u eps} Define $\bar{u}^{(\varepsilon )}:\Theta \rightarrow 
\mathbb{R}$ by 
\begin{equation*}
\bar{u}^{(\varepsilon )}(t,\omega ):=u^{(\varepsilon )}\big((\omega
_{t_{i}})_{0\leq i\leq k};t-t_{k},\omega _{t}-\omega _{t_{k}}\big),\quad %
\mbox{for}~~t\in \lbrack t_{k},t_{k+1}).
\end{equation*}%
There exist adapted processes $\l_t(\o),\varphi _{t}\left( \o\right) ,\eta_t \left(
\o\right) $ such that for all $\alpha \in \mathcal{A}$ 
\begin{equation*}
\bar{u}^{(\varepsilon )}(T,X^{\alpha })=\bar{u}_{0}^{(\varepsilon
)}+\int_{0}^{T}\Big(\l_t + \mu (t,\alpha _{t})\varphi_t +\frac{1}{2}\mathrm{%
Tr}\big((\sigma \sigma ^{\intercal})(t,\alpha _{t})\eta_t \big)
\Big) \big(X^{\alpha }\big)dt+\int_{0}^{T}\varphi _{t}(X^{\alpha })^{\intercal}\sigma (t,\alpha
_{t})dB_{t},
\end{equation*}
$\mathbb{P}_{0}\mbox{-a.s.},$ and%
\begin{equation*}
\Big(\l_t + \mu (t,\alpha_t )\varphi_t +\frac{1}{2}\mathrm{Tr}\big((\sigma
\sigma ^{\intercal})(t,\alpha_t )\eta_t \Big)(\omega )\leq 0,\quad 
\mbox{for
all}~~\alpha \in \mathcal{A},(t,\omega )\in \Theta .
\end{equation*}
\end{cor}
\begin{proof}
By It\^o's formula, we have
\begin{multline*}
\bar{u}^{(\varepsilon )}(t,X^{\alpha}) ~=~ \bar{u}^{(\varepsilon )}(t_k,X^{\alpha}) 
+\int_{t_k}^{t}\Big(\l_s+\mu (s,\alpha _{s})\varphi_s +\frac{1}{2}\mathrm{Tr}\big((\sigma \sigma ^{\intercal})(s,\alpha _{s})\eta_s \big)\Big) \big(X^{\alpha }\big)ds\\
+\int_{t_k}^{t}\varphi_{s}(X^{\alpha })^{\intercal}\sigma (s,\alpha
_{s})dB_{s},\q\mbox{for $t\in [t_k,t_{k+1})$,\q $\dbP_0$-a.s.},
\end{multline*}
with
\beaa
\left.\ba{lll}
& \l_t(\o) := \pa_t u^{(\e)}\big((\o_{t_i})_{0\le i\le k}; t -t_k, \o_t-\o_{t_k}\big), &\\
&\f_t(\o) := \pa_x u^{(\e)}\big((\o_{t_i})_{0\le i\le k}; t -t_k, \o_t-\o_{t_k}\big),&\\
&\eta_s(\o) := \pa^2_{xx} u^{(\e)}\big((\o_{t_i})_{0\le i\le k}; t -t_k, \o_t-\o_{t_k}\big),&
\ea\right.
\q\mbox{for $t\in [t_k,t_{k+1})$}.
\eeaa
By the supersolution property of $u^{(\varepsilon )}$ proved in Proposition \ref{prop u eps}, we have
\begin{multline*}
\Big(\l_t+\mu (t,\alpha _{t})\varphi_t +\frac{1}{2}\mathrm{Tr}\big((\sigma \sigma ^{\intercal})(t,\alpha _{t})\eta_t \big)\Big)(\o) \\ 
\le ~ {\bf L}^k u^{(\e)}\big((\o_{t_i})_{0\le i\le k}; \cd,\cd\big)(t -t_k, \o_t-\o_{t_k}) ~\le ~ 0.
\end{multline*}
The proof is completed.
\end{proof}

\no Finally, we prove an approximation analogous to Theorem \ref{thm: markov} in
our non-Markovian setting.

\begin{thm}
Suppose Assumption \ref{assum:non Markov} holds. Then we have 
\begin{equation*}
u_{0}=\lim_{h\rightarrow 0}v^{h},\q \mbox{where}~v^{h}:=\inf_{\varphi \in 
\mathcal{U}}\mathbb{E}^{\mathbb{P}_{0}}\Big[\sup_{a \in \mathcal{D}_{h}}%
\Big\{g(X_{T\wedge \cdot }^{a })-\int_{0}^{T}\varphi _{t}(X^{a })^\intercal \sigma (t,a _{t})dB_{t}\Big\}\Big].
\end{equation*}
\end{thm}

\begin{proof}
Arguing as in the proof of Theorem \ref{thm: markov}, one can easily deduce
using the Ito formula that $u_{0}\leq \lim_{h\rightarrow 0}v^{h}$.

Consider the function $\bar{u}^{(\varepsilon )}$ and let $\varphi $ be the
process defined in Corollary \ref{cor: u eps}. We have 
\begin{eqnarray*}
v^{h} &\leq &\mathbb{E}^{\mathbb{P}_{0}}\Big[\sup_{a \in \mathcal{D}_{h}}%
\Big\{g(X_{T\wedge \cdot }^{a })-\int_{0}^{T}\varphi _{t}(X^{a
})^{\intercal}\sigma (t,a _{t})dB_{t}\Big\}\Big] \\
&\leq &\mathbb{E}^{\mathbb{P}_{0}}\Big[\sup_{a \in \mathcal{D}_{h}}\Big\{%
g(X_{T\wedge \cdot }^{a })-\bar{u}_{T}^{(\varepsilon )}(X^{a })+\bar{u}%
_{0}^{(\varepsilon )}\Big\}\Big] \\
&\leq &\mathbb{E}^{\mathbb{P}_{0}}\Big[\sup_{a \in \mathcal{D}_{h}}\Big\{%
g(X_{T\wedge \cdot }^{a })-g^{\varepsilon }\big((X_{t_{i}}^{a })_{0\leq
i\leq n}\big)\Big\}\Big]+\theta _{0}^{\varepsilon }+\varepsilon.
\end{eqnarray*}%
For the last inequality, we use the fact that $\bar{u}_{0}^{(\varepsilon
)}=u^{(\varepsilon )}(0;0,0)=\theta _{0}^{\varepsilon }+\varepsilon $. Note
that there are only finite elements in the set $\mathcal{D}_{h}$. Therefore,
by Lemma \ref{lem: th eps u} 
\begin{eqnarray*}
&&\mathop{\overline{\rm lim}}_{\varepsilon \rightarrow 0}\Big(\mathbb{E}^{%
\mathbb{P}_{0}}\Big[\sup_{a \in \mathcal{D}_{h}}\Big\{g(X_{T\wedge \cdot
}^{a })-g^{\varepsilon }\big((X_{t_{i}}^{a })_{0\leq i\leq n}\big)\Big\}\Big]%
+\theta _{0}^{\varepsilon }+\varepsilon \Big) \\
&\leq &\mathop{\overline{\rm lim}}_{\varepsilon \rightarrow 0}\Big(\sum_{a
\in \mathcal{D}_{h}}\mathbb{E}^{\mathbb{P}_{0}}\big[\big|g(X_{T\wedge \cdot
}^{a })-g^{\varepsilon }\big((X_{t_{i}}^{a })_{0\leq i\leq n}\big)\big|\big]%
+\theta _{0}^{\varepsilon }+\varepsilon \Big) \\
&=&u_{0}.
\end{eqnarray*}%
We conclude that $v^{h}\leq u_{0}$ for all $h\in (0,1\wedge T].$
\end{proof}

\subsection{Example of a duality result for an American option\label{sec:4}}

In this subsection we give an indication how our approach may be extended to
American options. To this end we consider a toy model, in which the $d$%
-dimensional controlled diffusion $X^{\alpha }$ takes the particular form $%
X^{\alpha }:=\int_{0}^{\cdot }\alpha _{t}^{0}dt+\int_{0}^{\cdot }\alpha
_{t}^{1}dB_{t}$ and carry out the analysis in this elementary setting. The
stochastic control problem is now 
\begin{equation*}
u_{0}=\sup_{\alpha \in \mathcal{A},\tau \in \mathcal{T}_{T}}\mathbb{E}^{%
\mathbb{P}_{0}}\big[g(X_{\tau }^{\alpha })\big],
\end{equation*}%
where $\mathcal{T}_{T}$ is the set of all stopping times smaller than $T$.
Throughout this subsection we will make the following assumption:

\begin{assum}\label{assum: american}
Suppose $g:\dbR^d\rightarrow\dbR$ to be bounded and uniformly continuous.
\end{assum}

For $\alpha \in \mathcal{A}$ define probability measures $\mathbb{P}_{\alpha
}:=\mathbb{P}_{0}\circ (X^{\alpha })^{-1}$, let $\mathcal{P}:=\{\mathbb{P}%
_{\alpha }:\alpha \in \mathcal{A}\}$ and define the nonlinear expectation $%
\mathcal{E}[\cdot ]:=\sup_{\mathbb{P}\in \mathcal{P}}\mathbb{E}^{\mathbb{P}%
}[\cdot ]$. It will be convenient to use the shorthand $\alpha ^{1}\cdot B$
for the stochastic integral $\int_{0}^{\cdot }\alpha _{s}^{1}dB_{s}.$ We
have 
\begin{equation*}
u_{0}=\sup_{\tau \in \mathcal{T}_{T}}\mathcal{E}\big[g(B_{\tau })\big].
\end{equation*}%
Further, we define the dynamic version of the control problem: 
\begin{equation*}
u(t,x):=\sup_{\tau \in \mathcal{T}_{T-t}}\mathcal{E}\big[g(x+B_{\tau })%
\big],\quad \mbox{for}~~(t,x)\in \lbrack -1,T]\times \mathbb{R}^{d}.
\end{equation*}%
The following lemma shows that the function $u$ satisfies a dynamic
programming principle (see for example Lemma 4.1 of \cite{ETZ-os} for a
proof).

\begin{lem}
\label{lem: yet another}The value function $u$ is continuous in both
arguments, and we have 
\begin{equation*}
u(t_{1}, x)=\sup_{\tau \in \mathcal{T}_{T-t_{1}}}\mathcal{E}\big[%
g(x+B_{\tau })1_{\{\tau <t_{2}\}}+u(t_{2}, x+B_{t_2})1_{\{\tau \geq t_{2}\}}\big].
\end{equation*}%
In particular, $\{u(t,B_t)\}_{t\in [0,T]}$ is a $\mathbb{P}$-supermartingale for all $\mathbb{P}\in 
\mathcal{P}$.
\end{lem}

Next we apply the familiar mollification technique already employed in
Section \ref{subsec:markov}. Define $u^{(\varepsilon )}:=u\ast
K^{(\varepsilon )}$.

\begin{lem}\label{lem: u e} 
$\{u^{(\varepsilon )}(t,B_{t})\}_{t}$ is a $\mathbb{P}$%
-supermartingale for all $\mathbb{P}\in \mathcal{P}$, and $u^{(\varepsilon
)}\geq g^{(\varepsilon )}:=g\ast K^{(\varepsilon )}$.
\end{lem}

\begin{proof}
For any $s\leq t\leq T$ and $x\in \mathbb{R}$, we have by Lemma \ref{lem:
yet another} 
\begin{eqnarray*}
\mathcal{E}\big[u^{(\varepsilon )}(t,x+B_{t-s})\big] &=&\mathcal{E}\Big[\int
u(t-r,x-y+B_{t-s})K^{(\varepsilon )}(r,y)dydr\Big] \\
&\leq &\int \mathcal{E}\big[u(t-r,x-y+B_{t-s})\big]K^{(\varepsilon
)}(r,y)dydr \\
&\leq &\int u(t-r-(t-s),x-y)K^{(\varepsilon )}(r,y)dydr \\
& = & \int u(s-r,x-y)K^{(\varepsilon )}(r,y)dydr ~=~u^{(\varepsilon )}(s,x),
\end{eqnarray*}%
where for the second inequality, we used the $\dbP$-supermartingale property of $\{u(t, B_t)\}_{t\in [0,T]}$ for all $\dbP\in \cP$.
This implies that for all $\mathbb{P}\in \mathcal{P}$ we have 
\begin{equation*}
\mathbb{E}^{\mathbb{P}}\big[u^{(\varepsilon )}(t,x+B_{t-s})\big]\leq
u^{(\varepsilon )}(s,x).
\end{equation*}%
Therefore, $\{u^{(\varepsilon )}(t,B_{t})\}_{t}$ is a $\mathbb{P}$%
-supermartingale for all $\mathbb{P}\in \mathcal{P}$. On the other hand, it
is clear from the definition of $u$ that $u\geq g$ and, hence, $%
u^{(\varepsilon )}\geq g^{(\varepsilon )}$.
\end{proof}

Again, the stochastic control problem can be discretized.  For technical reasons, we assume here that the partitions of time satisfy the order:
\be\label{orderpartition}
\{t^h_i\}_{i\le n_h} \subset \{t^{h'}_i\}_{i\le n_{h'}}, \q\mbox{for}~~h>h',
\ee
where $n_h$ is the number of the time grids of the partition.

\begin{lem}
\label{lem convergence}Under Assumption \ref{assum: american}, it holds 
\begin{equation}
u_{0}=\lim_{h\rightarrow 0}u_{0}^{h},\quad \mbox{where}~u_{0}^{h}:=\sup_{%
\alpha \in \mathcal{A}_{h},\tau \in \mathcal{T}_{T}}\mathbb{E}^{\mathbb{P}%
_{0}}\Big[g(X_{\tau }^{\alpha })\Big].  \label{n-pbm american}
\end{equation}
\end{lem}

\begin{proof}
We only prove the case $\alpha ^{0}=0$ and $\alpha =\alpha ^{1}$ $\in $ $A^1$, a compact set in $\dbR$, in particular, $X^{\alpha }=(\alpha \cdot B).$ The general case follows
by a straightfoward generalization of the same arguments. Note that it is
sufficient to show that $u_{0}\leq \mathop{\underline{\rm lim}}%
_{h\rightarrow 0}u_{0}^{h}$. Fix $\epsilon >0.$ There exists $\alpha
^{\varepsilon }\in \mathcal{A}$ such that 
\begin{equation}
u_{0}<\sup_{\tau \in \mathcal{T}_{T}}\mathbb{E}^{\mathbb{P}_{0}}\big[g\big(%
(\alpha ^{\varepsilon }\cdot B)_{\tau }\big)\big]+\varepsilon .
\label{eq:eps optimal u0}
\end{equation}%
For any $h$ sufficiently small define a process $\tilde{\alpha}^{h}$ by
letting 
\begin{equation*}
\tilde{\alpha}_{t}^{h}:=\sum_{i}\frac{1}{t_{i+1}^{h}-t_{i}^{h}}%
\int_{t_{i}^{h}}^{t_{i+1}^{h}}\mathbb{E}^{\mathbb{P}_{0}}\left[
\alpha _{s}^{\varepsilon }\big| \cF_{t_{i}^{h}} \right] ds\mathbf{1}_{[t_{i}^{h},t_{i+1}^{h})}(t).
\end{equation*}%
Clearly, $\tilde{\alpha}^{h}$ is piecewise constant on each interval $%
[t_{i}^{h},t_{i+1}^{h})$. We introduce the filtration $\hat\dbF := \{ \hat\cF_h \}_h$, with
\beaa
\hat\cF_h := \si\Big(\Big\{[t^h_i, t^h_{i+1})\times A: i\le n_h-1, A\in \cF_{t^h_i} \Big\}\Big).
\eeaa
In particular, it follows from \eqref{orderpartition} that $\hat\cF_h \subset \hat \cF_{h'}$ for $h>h'$. Also, denote the probability $\hat \dbP$ on the product space $\Th$:
\beaa
\hat\dbP(dt,d\o) := \frac{1}{T}dt \times \dbP_0(d\o). 
\eeaa
Note that for all $i\le n_h-1$, $A\in \hat\cF_{t^h_i}$ and $h'< h$ we have
\beaa
\dbE^{\hat \dbP_0}\Big[\tilde \a^{h'} ~1_{\{[t^h_i,t^h_{i+1})\times A\}}\Big]
&=& \dbE^{\dbP_0}\left[\frac{1}{T}\sum_{j: t^{h}_i\le t^{h'}_j, t^{h'}_{j+1}\le t^h_{i+1}}
\int_{t^{h'}_j}^{t^{h'}_{j+1}} \mathbb{E}^{\mathbb{P}_{0}}\left[
\alpha_{s}^{\varepsilon }\big| \cF_{t_{j}^{h'}} \right] ds ~1_A \right]\\
&=& \dbE^{\dbP_0}\left[\frac{1}{T}\int_{t^{h}_i}^{t^{h}_{i+1}} \mathbb{E}^{\mathbb{P}_{0}}\left[
\alpha_{s}^{\varepsilon }\big| \cF_{t_{i}^{h}} \right] ds ~1_A \right]\\
&=& \dbE^{\hat \dbP_0}\Big[\tilde \a^{h} ~1_{\{[t^h_i,t^h_{i+1})\times A\}}\Big].
\eeaa
So $\{\tilde \a^h\}_h$ is a martingale in the filtrated probability space $\big(\Th, \hat\dbP, \hat\dbF\big)$.
Note that $\a^\e$ and $\tilde \a^h$ are bounded, so it follows from the martingale
convergence theorem that
\begin{equation}
\lim_{h\rightarrow 0}\mathbb{E}^{\mathbb{P}_{0}}\int_{0}^{T}(\alpha
_{s}^{\varepsilon }-\tilde{\alpha}_{s}^{h})^{2}ds~=~0  \label{eq:ah limit}
\end{equation}
Further, define $\hat{\alpha}^{h}:=h\left\lfloor 
\frac{\tilde{\alpha}^{h}}{h}\right\rfloor $ and note that we have $\hat{%
\alpha}^{h}\in \mathcal{A}_{h}$. It follows from \eqref{eq:ah limit} that
\begin{equation*}
\lim_{h\rightarrow 0}\mathbb{E}^{\mathbb{P}_{0}}\int_{0}^{T}(\alpha
_{s}^{\varepsilon }-\hat{\alpha}_{s}^{h})^{2}ds=0.
\end{equation*}%
With $\rho $ an increasing and concave modulus of continuity of $g$ we have 
\begin{eqnarray}
&& \sup_{\tau \in \mathcal{T}_{T}}\mathbb{E}^{\mathbb{P}_{0}}\big[g\big((\alpha
^{\varepsilon }\cdot B)_{\tau }\big)\big]-\sup_{\tau \in \mathcal{T}_{T}}%
\mathbb{E}^{\mathbb{P}_{0}}\big[g\big((\hat{\alpha}^{h}\cdot B)_{\tau }\big)%
\big] \notag \\
&\leq & \sup_{\t\in \cT_T}\mathbb{E}^{\mathbb{P}_{0}}\Big[\rho \big( | (\alpha
^{\varepsilon }\cdot B)_\t -(\hat{\alpha}^{h}\cdot B)_\t | \big)\Big]  \notag \\
&\le &   \mathbb{E}^{\mathbb{P}_{0}}\Big[\rho \big( \Vert (\alpha
^{\varepsilon }\cdot B) -(\hat{\alpha}^{h}\cdot B) \Vert_\infty \Big]  \notag \\
&=&\rho \Big(\mathbb{E}^{\mathbb{P}_{0}}\Big[\int_{0}^{T}(\alpha
_{s}^{\varepsilon }-\hat{\alpha}_{s}^{h})^{2}ds\Big]^\frac12 \Big)
\label{eqn yet another}
\end{eqnarray}%
Combining \eqref{eq:eps optimal u0}, \eqref{eqn yet another} we have%
\begin{eqnarray*}
u_{0} &<&\sup_{\tau \in \mathcal{T}_{T}}\mathbb{E}^{\mathbb{P}_{0}}\left[
g\left( (\hat{\alpha}^{h}\cdot B)_{\tau }\right) \right] +\rho \Big(\mathbb{E%
}^{\mathbb{P}_{0}}\Big[\int_{0}^{T}(\alpha _{s}^{\varepsilon }-\hat{\alpha}%
_{s}^{h})^{2}ds\Big]^\frac12 \Big)+\varepsilon \\
&\leq &u_{0}^{h}+\rho \Big(\mathbb{E}^{\mathbb{P}_{0}}\Big[%
\int_{0}^{T}(\alpha _{s}^{\varepsilon }-\hat{\alpha}_{s}^{h})^{2}ds\Big]^\frac12 \Big)%
+\varepsilon .
\end{eqnarray*}%
Letting $h\rightarrow 0$ we deduce 
\begin{equation*}
u_{0}\leq \underline{\lim }_{h\rightarrow 0}u_{0}^{h}+\varepsilon .
\end{equation*}%
for all $\varepsilon >0.$
\end{proof}

We conclude the section by proving the analogous approximation result for
American options.

\begin{thm}
\label{thm: american} Suppose Assumption \ref{assum: american} holds. Then
we have 
\begin{equation*}
u_{0}=\lim_{h\rightarrow 0}v^{h},\q \mbox{where}~v^{h}:=\inf_{\varphi \in 
\mathcal{U}}\mathbb{E}^{\mathbb{P}_{0}}\Big[\sup_{\alpha \in \mathcal{D}%
_{h},t\in \lbrack 0,T]}\Big\{g(X_{t}^{\alpha })-\int_{0}^{t}\varphi _{s}(X^{\alpha })^\intercal\alpha _{s}dB_{s}\Big\}\Big].
\end{equation*}
\end{thm}

\begin{proof}
We first prove that the left hand side is smaller. Recall $u_{0}^{h}$
defined in \eqref{n-pbm american}. For all $\varphi \in \mathcal{U}$, the
process $\int_{0}^{\cdot }\varphi _{t}(X^{\alpha })^{\intercal}\alpha
_{t}^{1}dB_{t}$ is a martingale, and we have 
\beaa
u_{0}^{h} &\leq &\sup_{\alpha \in \mathcal{A}_{h},\tau \in \mathcal{T}_{T}}%
\mathbb{E}^{\mathbb{P}_{0}}\Big[g(X_{\tau }^{\alpha })-\int_{0}^{\tau
}\varphi _{t}(X^{\alpha })^{\intercal}\alpha _{t}^{1}dB_{t}\Big],~~\mbox{for all}~\varphi \in \mathcal{U}.
\eeaa
 Since for any $\a\in \cA_h$ and $\t\in \cT_T$ we have $g(X_{\tau }^{\alpha })-\int_{0}^{\tau
}\varphi _{t}(X^{\alpha })^{\intercal}\alpha _{t}^{1}dB_{t} \le \sup_{a \in \mathcal{D}_{h},t\in
\lbrack 0,T]}\Big\{g(X_{t}^{a })-\int_{0}^{t}\varphi _{s}(X^{a
})^{\intercal}a _{s}^{1}dB_{s}\Big\}$, we obtain
\beaa
u_0^h &\leq &\mathbb{E}^{\mathbb{P}_{0}}\Big[ \sup_{a \in \mathcal{D}_{h},t\in
\lbrack 0,T]}\Big\{g(X_{t}^{a })-\int_{0}^{t}\varphi _{s}(X^{a
})^{\intercal}a _{s}^{1}dB_{s}\Big\}\Big],~~\mbox{for all}~\varphi \in \mathcal{U}.
\end{eqnarray*}%
The required result follows by Lemma \ref{lem convergence}. 

For the converse, recall that $u^{(\varepsilon
)}(t,B_{t})$ is a $\mathbb{P}$-supermartingale for all $\mathbb{P}\in 
\mathcal{P}$ (Lemma \ref{lem: u e}). Further, since $u^{(\varepsilon )}\in C^{1,2}$, we have 
\begin{equation*}
\partial _{t}u^{(\varepsilon )}+\sup_{(b^0,b^1) \in A}\Big\{ b^{0}\partial
_{x}u^{(\varepsilon )}+\frac{1}{2}\mathrm{Tr}\big(b^{1}(b^{1})^\intercal \partial _{xx}^{2}u^{(\varepsilon )}\big)\Big\}\leq 0.
\end{equation*}
Hence, for all $h>0$ 
\begin{eqnarray*}
v_{h} &\leq &\mathbb{E}^{\mathbb{P}_{0}}\Big[\sup_{a \in \mathcal{D}%
_{h},t\in \lbrack 0,T]}\Big\{g(X_{t}^{a })-\int_{0}^{t}\partial
_{x}u_{s}^{(\varepsilon )}(X^{a })^{\intercal}a _{s}^{1}dB_{s}\Big\}\Big]
\\
&\leq &\mathbb{E}^{\mathbb{P}_{0}}\Big[\sup_{a \in \mathcal{D}_{h},t\in
\lbrack 0,T]}\Big\{g(X_{t}^{a })-u_{t}^{(\varepsilon )}(X_{t}^{a
})+u_{0}^{(\varepsilon )} \\
&&\quad \quad \quad +\int_{0}^{t}\Big(\partial _{t}u_{s}^{(\varepsilon
)}(X_{s}^{a })+a _{s}^{0}\cdot \partial _{x}u_{s}^{(\varepsilon )}(X_{s}^{a
})+\frac{1}{2}\mathrm{Tr}\big(a _{s}^{1}(a _{s}^{1})^\intercal \partial
_{xx}^{2}u_{s}^{(\varepsilon )}(X_{s}^{a })\big)\Big)ds\Big\}\Big] \\
&\leq &\mathbb{E}^{\mathbb{P}_{0}}\Big[\sup_{a \in \mathcal{D}_{h},t\in
\lbrack 0,T]}\big\{g(X_{t}^{a })-g^{(\varepsilon )}(X_{t}^{a })\big\}\Big]%
+u_{0}^{(\varepsilon )},
\end{eqnarray*}%
where we have used Ito's formula and the inequality $u^{(\varepsilon )}\geq
g^{(\varepsilon )}$ proved in Lemma \ref{lem: u e}. It is straightforward to
check that%
\begin{equation*}
\lim_{\varepsilon \rightarrow 0}\left( \mathbb{E}^{\mathbb{P}_{0}}\Big[%
\sup_{a \in \mathcal{D}_{h},t\in \lbrack 0,T]}\big\{g(X_{t}^{a
})-g^{(\varepsilon )}(X_{t}^{a })\big\}\Big]+u_{0}^{(\varepsilon )}\right)
=u_{0}.
\end{equation*}
\end{proof}

\section{Examples}

\subsection{Uncertain volatility model\label{sec:5}}

\noindent As a first example, we consider an uncertain volatility model
(UVM), first considered in \cite{ALP} and \cite{L}. Let $A\subseteq \mathbb{R}%
^{d}\times \mathbb{R}^{d\times d}$ be a compact domain such that for all $\left( \sigma ^{i}\,,\rho ^{ij}\right) _{1\leq i,j\leq d}$ $\in A$, the
matrix 
\begin{equation*}
\left( \rho ^{ij}\sigma ^{i}\sigma ^{j}\right) _{1\leq i,j\leq d}
\end{equation*}%
is positive semi-definite, $\rho ^{ij}=\rho ^{ji}\in \left[ -1,1\right] $
and $\rho ^{ii}=1.$ If $d=2$ an example of such a domain is obtained by
setting 
\begin{equation*}
A ~=~ \Big(\prod_{i=1}^{2}[\underline{\sigma }^{i},\overline{\sigma }^{i}]\Big)%
\times \left\{ \left( 
\begin{array}{cc}
1 & \rho \\ 
\rho & 1%
\end{array}%
\right) :\rho \in \left[ \underline{\rho },\overline{\rho }\right] \right\} ,
\end{equation*}%
where $0\leq \underline{\sigma }^{i}\leq \overline{\sigma }^{i}$ and $-1\leq 
\underline{\rho }\leq \overline{\rho }\leq 1$. Recall the definition of $\cA$, i.e. an adapted process $\left(
\sigma ,\rho \right) =\left( \sigma _{t},\rho _{t}\right) _{0\leq t\leq T}\in \cA$ if it takes
values in $A$.
In the UVM the stock prices follow the dynamics%
\begin{equation*}
d(X^{\si,\rho}_{t})^{i}=\sigma _{t}^{i}(X^{\si,\rho}_{t})^{i}dW_{t}^{i},\quad d\langle
W^{i},W^{j}\rangle _{t}=\rho ^{ij}dt,\quad 1\leq i<j\leq d,
\end{equation*}%
where $W^i$ is a $1$-dimensional Brownian motion for all $i\le d$, and  $\left( \sigma ,\rho \right) \in \cA$ is the unknown volatility
process and correlation. The value of the option at time $t$ in the UVM,
interpreted as a super-replication price under uncertain volatilities, is
given by 
\begin{equation}\label{valuefunUVM}
u_{t} ~=~ \sup_{\left( \sigma ,\rho \right) \in \cA} \mathbb{E}\big[\xi_{T}(X^{\si,\rho})|\mathcal{F}_{t}\big].
\end{equation}

\noindent For European payoffs, $\xi_{T}(\o)=g(\o_{T})$, the value $u(t,x)$ is
then the unique viscosity solution (under suitable conditions on $g$)
of the nonlinear PDE: 
\be\label{eq:HJBuvm}
\partial _{t}u(t,x)+H(x,D_{x}^{2}u(t,x))=0,\q u(T,x)=g(x)
\ee
with the Hamiltonian 
\begin{equation*}
H(x,\gamma )={\frac{1}{2}}\max_{(\sigma ^{i},\rho ^{ij})_{1\leq i,j\leq
d}\in A}\sum_{i,j=1}^{d}\rho ^{ij}\sigma ^{i}\sigma ^{j}x^{i}x^{j}\gamma ^{ij},
\q
\mbox{for all}~~ x\in \dbR^d, \g \in \dbR^{d\times d}.
\end{equation*}

\ms

\subsubsection*{Second order backward stochastic differential equation (2BSDE)} 

Fix constants $\hat \si = (\hat \si^i)_{1\le i\le d}$ and $\hat \rho = (\hat\rho^{i,j})_{1\le i,j\le d}$. Denote a new diffusion process $\widehat X$:
\beaa
d\widehat{X}_{t}^{i} &=& \hat{\sigma }^{i}\widehat{X}_{t}^{i}d\widehat W_{t}^{i},
\q d\langle \widehat W^{i}, \widehat W^{j}\rangle_t=\hat{\rho }^{ij}dt,\text{ }1\leq i\leq
j\leq d,
\eeaa
where $\widehat W^i$ is $1$-dimensional Brownian motion for all $1\le i\le d$. Consider the dynamics: 
\be\label{dym2bsde}
\left.\ba{lll}
dZ_{t} &=&\Xi_{t} dt+ \Gamma_t d \widehat X_t,  \\
dY_{t} &=&-H\big( \widehat{X}_{t},\Gamma _{t}\big)dt+Z_{t} d\widehat{X}_t 
 + \frac12 \big(\hat\si \widehat X_t \big)^\intercal \Gamma_t \big(\hat\si\widehat X_t \big)dt,
\ea\right.
\ee
where  $(Y,Z,\Gamma, \Xi)$ is a quadruple taking values in $\mathbb{R},$ $\mathbb{R}^{d},S^{d}$ (the space of symmetric $d\times d$ matrices) and $\mathbb{R}^{d}$ respectively. In particular, if the HJB equation \eqref{eq:HJBuvm} has a smooth solution, it follows from the It\^o's formula that
\be\label{eq:PDE2BSDE}
Y_t:= u(t, \widehat X_t),\q Z_t := \pa_x u(t, \widehat X_t), \q \Gamma_t:= \pa^2_{xx} u(t, \widehat X_t)
\ee
satisfy the dynamics \eqref{dym2bsde} with a certain process $\Xi$. In Cheridito, Soner, Touzi \& Victoir \cite{CSTV}, the authors studied the existence and uniqueness of the quadruple  $(Y,Z,\Gamma, \Xi)$ satisfying the dynamics \eqref{dym2bsde} with the terminal condition $Y_T = g(\widehat X_T)$, without assuming the existence of smooth solution to the HJB equation \eqref{eq:HJBuvm}, and they gave the name `2BSDE' to this problem. For the readers interested in the theory of 2BSDE, we refer to \cite{CSTV} and Soner, Touzi \& Zhang \cite{STZ} for more details.

\ms

\subsubsection*{Numerical scheme for 2BSDE}

We are interested in solving the 2BSDE numerically. In the existing literature, one may find several different numerical schemes for this problem (see for example \cite{CSTV, fah, guy}). Here we recall the one proposed in Guyon \& Henry-Labord\`ere \cite{guy}. Introduce the partition $\{t_i\}_{i\le n}$ on the interval $\left[ 0,T\right] $, and denote $\Delta t_{i}=t_{i}-t_{i-1}$, $\Delta
W_{t_{i}}=W_{t_{i}}-W_{t_{i-1}}$. First, the diffusion $\widehat X$ can be written explicitly:
\beaa
\widehat{X}_{t_{i}}^{j}=\widehat{X}_{0}^{j}e^{-({\hat{\sigma}}^{j})^{2}{%
\frac{t_{i}}{2}}+{\hat{\sigma}}^{j}W_{t_{i}}^{j}}\;,\q\mbox{with}~~\Delta
W_{t_{i}}^{j}\Delta W_{t_{i}}^{k}=\hat{{\rho }}_{jk}\Delta t_{i} .
\eeaa
Denote by $\widehat Y, \widehat \Gamma$ the numerical approximations of $Y, \Gamma$. In the backward scheme in \cite{guy}, we set $\widehat Y_{t_{n}} ~=~ g\big(\widehat{X}_{t_{n}}\big)$, and then compute
\beaa
&\hat{\sigma}^{j}\hat{\sigma}^{k}\widehat{X}_{0}^{j}\widehat{X}_{0}^{k}{\
\widehat\Gamma _{t_{i-1}}^{jk}}
~=~
\mathbb{E}_{i-1} \Big[\widehat Y_{t_{i}}\big(U_{t_{i}}^{j}U_{t_{i}}^{j}-(\Delta t_{i})^{-1}\hat{\rho}_{jk}^{-1}-\hat{%
\sigma}^{j}U_{t_{i}}^{j}\delta _{jk}\big) \Big] & \notag\\
& \mbox{with}\q U_{t_{i}}^{j} ~:=~ \sum_{k=1}^{d}\hat{\rho}_{jk}^{-1}\Delta
W_{t_{i}}^{k}/\Delta t_{i},  \q\mbox{and} &\\
& \widehat Y_{t_{i-1}}
~=~
 \mathbb{E}_{i-1} \big[\widehat Y_{t_{i}}\big]+\Big( H(%
\widehat{X}_{t_{i-1}},{\ \widehat\Gamma _{t_{i-1}}})-{\frac{1}{2}}\sum_{j,k=1}^{n}%
\widehat{X}_{t_{i-1}}^{j}\widehat{X}_{t_{i-1}}^{k}{\Gamma _{t_{i-1}}^{jk}}%
\hat{\rho}_{pk}\hat{\sigma}^{j}\hat{\sigma}^{k}\Big) \Delta t_{i},& \notag
\eeaa
where $\dbE_i$ denotes the conditional expectation with respect to the filtration $\cF_{t_i}$. Below, we denote $u_{0}^{\mathrm{BSDE}} := \widehat Y_0$.

\ms

\subsubsection*{Lower and upper bound for the value function}

Once $\widehat\Gamma$ is computed, one gets a (sub-optimal) estimation of the controls $(\hat\sigma ^*,\hat\rho^*)$:
\beaa
\big(\hat\sigma^*_{t_i},\hat\rho^*_{t_i}\big)~:=~ {\rm argmax}_{(\sigma ^{j},\rho ^{jk})_{1\leq j,k \leq
d}\in A}\sum_{j,k=1}^{d}\rho ^{jk}\sigma ^{j}\sigma ^{k} \widehat X^{j}_{t_i} \widehat X^{k} _{t_i} \widehat\Gamma^{jk}_{t_i},\q \mbox{for}~~0\le i\le n.
\eeaa
 Performing a second independent
(forward) Monte-Carlo simulation using this sub optimal control, we obtain a lower
bound for the value function \eqref{valuefunUVM}:
\beaa
u_{0}^{\mathrm{LS}}
~:= ~  \dbE\big[g(X^{\hat\si^*,\hat\rho^*}_T)\big]
~\leq~ {u}_{0}.
\eeaa 

We next calculate the dual bound derived in the current paper. As mentioned in Remark \ref{rem:1}, we will use the numerical approximation of $\pa_{x}u$ to serve
as the minimizer $\varphi ^{\ast }$ in the dual form. Also, we observe from \eqref{eq:PDE2BSDE} that the process $Z$ in the 2BSDE plays the corresponding role of $\pa_x u$, and we can compute the numerical approximation $\widehat Z$ of $Z$:
\begin{equation*}
\hat{\sigma}^{j}\widehat{X}_{t_i}^{j}\widehat Z^{j}_{t_i}={\mathbb{E}}_{i-1} \big[\widehat Y_{t_{i}}U_{t_{i}}^{j} \big].
\end{equation*}%
Then we define
\begin{equation*}
\varphi _{t}^{\ast }=\sum_{i=1}^{n} \widehat Z_{t_{i-1}} 1_{ \lbrack t_{i-1},t_{i})}(t).
\end{equation*}%
 Using our candidate $\varphi
^{\ast }$ in the minimization, we get an upper bound 
\begin{equation*}
u_{0}^{\mathrm{LS}}
~\leq~
 u_{0}
 ~\leq~ u_{0}^{\mathrm{dual}} 
 ~:=~
\lim_{h\rightarrow 0}{\mathbb{E}}\Big[\max_{(\sigma ,\rho )\in \mathcal{D}_{h}}
\Big\{g(X^{\si,\rho}_{t_{n}})
-\sum_{i=1}^{n}\varphi _{t_{i-1}}^{\ast }(X^{\si,\rho})( X^{\si,\rho}_{t_{i}} - X^{\si,\rho}_{t_{i-1}})\Big\}\Big].
\end{equation*}

\ms

\subsubsection*{The algorithm}

 Our whole algorithm can be summarized by the following four steps:

\begin{enumerate}
\item Simulate $N_1$ replications of $\widehat X$ with a lognormal diffusion (we
choose $\hat{\sigma}=(\overline{\sigma}+\underline{\sigma})/2$).

\item Apply the backward algorithm using a regression approximation. Compute $%
Y_{0}=u_{0}^{\mathrm{BSDE}}.$

\item Simulate $N_2$ independent replication of $X^{\hat\si^*,\hat\rho^*}$ using the sub-optimal
control $(\hat\si^*,\hat\rho^*)$. Give a low-biased estimate $u_0^{\mathrm{LS}}$.

\item Simulate independent increment $\Delta W_{t_{i}}$ and maximize 
$g(X^{\si,\rho}_{t_{n}})-\sum_{i=1}^{n}\varphi_{t_{i-1}}^{\mathrm{\ast}}(X^{\si,\rho})(X^{\si,\rho}_{t_{i}}-X^{\si,\rho}_{t_{i-1}})$
over $(\sigma,\rho)\in \cD_h$. In our numerical
experiments, as the payoff may be non-smooth, we have used a direct search
polytope algorithm. Then compute the average.
\end{enumerate}

\subsubsection*{Numerical experiments}

  In our experiments, we take $T=1$ year and for the $i$-th asset, $X_{0}^{i}=100$, $\underline{\sigma }^{i}=0.1$, $\overline{%
\sigma }^{i }=0.2$ and we use the constant mid-volatility $\hat{\sigma}^{i}=0.15$ to generate the first $N_{1}=2^{15}$ replication of $\widehat X$. For the
second independent Monte-Carlo using our sub-optimal control, we take $N_2=2^{15}$ replications of $X$ and a time step $\Delta _{\mathrm{LS%
}}=1/400$. In the backward and dual algorithms, we choose the time step $\Delta$ among $\{1/2, 1/4, 1/8, 1/12\}$, which gives the biggest $u_{0}^{\mathrm{LS}}$ and the smallest $u_{0}^{\mathrm{dual}}$. The conditional
expectations at $t_{i}$ are computed using parametric regressions. The
regression basis consists in some polynomial basis. The exact price is obtained by
solving the (one or two-dimensional) HJB equation with a finite-difference
scheme.

\begin{enumerate}
\item $90-110$ call spread $(X_{T}-90)^{+}-(X_{T}-110)^{+}$, basis= $5$%
-order polynomial: 
\begin{equation*}
u_{0}^{\mathrm{LS}}=11.07<u_{0}^{\mathrm{PDE}}=11.20<u_{0}^{\mathrm{dual}%
}=11.70,\;u_{0}^{\mathrm{BSDE}}=10.30
\end{equation*}

\item Digital option $1_{X_{T}\geq 100}$, basis= $5$-order polynomial: 
\begin{equation*}
u_{0}^{\mathrm{LS}}=62.75<u_{0}^{\mathrm{PDE}}=63.33<u_{0}^{\mathrm{dual}%
}=66.54,\;u_{0}^{\mathrm{BSDE}}=52.03
\end{equation*}

\item Outperformer option $(X_{T}^{2}-X_{T}^{1})^{+}$ with 2 uncorrelated
assets, 
\begin{equation*}
u_{0}^{\mathrm{LS}}=11.15<u_{0}^{\mathrm{PDE}}=11.25<u_{0}^{\mathrm{dual}%
}=11.84,\;u_{0}^{\mathrm{BSDE}}=11.48
\end{equation*}

\item Outperformer option with 2 correlated assets $\rho =-0.5$ 
\begin{equation*}
u_{0}^{\mathrm{LS}}=13.66<u_{0}^{\mathrm{PDE}}=13.75<u_{0}^{\mathrm{dual}%
}=14.05,\;u_{0}^{\mathrm{BSDE}}=14.14
\end{equation*}

\item Outperformer spread option $%
(X_{T}^{2}-0.9X_{T}^{1})^{+}-(X_{T}^{2}-1.1X_{T}^{1})^{+}$ with 2 correlated
assets $\rho =-0.5$, {\ 
\begin{equation*}
u_{0}^{\mathrm{LS}}=11.11<u_{0}^{\mathrm{PDE}}=11.41<u_{0}^{\mathrm{dual}%
}=12.35,\;u_{0}^{\mathrm{BSDE}}=9.94
\end{equation*}%
}
\end{enumerate}
In examples 3.-5. the regression basis we used consists
of 
\begin{equation*}
\{1,X^{1},X^{2},(X^{1})^{2},(X^{2})^{2},X^{1}X^{2}\}.
\end{equation*}

\begin{rem}
The dual bounds we have derived complement the lower bounds derived in \cite%
{guy}. They allow us to access the quality of the regressors used in
computing the conditional expectations.
\end{rem}

\subsection{Credit value adjustment\label{sec:6}}

Our second example arises in credit valuation adjustment. We will show that
for this particular example, we can solve the deterministic optimization
problems arising in the dual algorithm efficiently by recursively solving
ODEs. 

\subsubsection*{CVA interpretation}

Let us recall the problem of the unilateral counterparty value adjustement (see \cite{guyPHL} for more details). We have one counterparty, denoted by C, that may
default and another, B, that cannot. We assume that B is allowed to trade
dynamically in the underlying $X$ - that is described by a local martingale 
\beaa
dX_{t}=\sigma (t,X_{t})dW_{t},\q\mbox{with $W$ a Brownian motion,}
\eeaa
 under a risk-neutral measure.
The default time of C is modeled by an exponential variable $\t$ with a
intensity $c$, independent of $W$. We denote by $u_0$ the value at time $0$ of B's long position in a single
derivative contracted by C, given that C has not defaulted so far. For
simplicity, we assume zero rate.  Assume that $g(X_T)$ is the payoff of the derivative at maturity $T$, and that $\tilde{u}$ is the derivative value just after the counterparty has
defaulted. Then, we have
\beaa
u_0 &=& \dbE\Big[g(X_T) 1_{\{\t >T\}} + \tilde u(\t,X_\t) 1_{\{\t \le T\}} \Big]\\
	&=&  \dbE\Big[ e^{-cT} g(X_T) + \int_0^T \tilde u(t,X_t) c e^{-ct} dt \Big].
\eeaa
Write down the dynamic version:
\beaa
u(t,x) &=&  \dbE\Big[ e^{-c(T-t)} g(X_T) + \int_t^T e^{-c(s-t)} c\tilde u(s,X_s) ds \Big| X_t =x \Big]. 
\eeaa
The function $u$ can be characterized by the equation: 
\begin{equation*}
\partial _{t}u+\frac{1}{2}\sigma ^{2}(t,x)\partial _{xx}^{2}u+c\left( \tilde{%
u}-u\right) =0, \q u(T,x) = g(x).
\end{equation*}%
 At the default event, in the case of zero recovery, we assume that $\tilde{u}$ is given by 
 \begin{equation*}
\tilde{u} ~=~ u^{-},
\end{equation*}%
where $x^-:= \max (0,-x)$.
Indeed, if the value of $u$ is positive, meaning that $u$ should be paid by
the counterparty, nothing will be received by B after the default. If the
value of $u$ is negative, meaning that $u$ should be received by the
counterparty, B will pay $u$ in the case of default of $C$. 

\begin{rem}
The funding value adjustment (FVA) corresponds to a similar nonlinear equation. 
\end{rem}

By the following change of variable
\begin{equation*}
u(t,x)^{\mathrm{HJB}}=e^{c(T-t)}u (t,x),
\end{equation*}
the function $u^{\mathrm{HJB}}$ satisfies the HJB equation:
\begin{equation}\label{HJB UVM}
\partial _{t}u^{\mathrm{HJB}}+\frac{1}{2}\sigma ^{2}(t,x)\partial
_{xx}^{2}u^{\mathrm{HJB}}+c(u^{\mathrm{HJB}})^{-}=0,\q u^{\rm HJB}(T,x) =g(x).
\end{equation}
The stochastic representation is:
\beaa
u^{\mathrm{HJB}}(t,x)=\sup_{\a \in \cA}{\mathbb{E}}\Big[e^{-\int_{t}^{T} \a_{s} ds}g(X_{T}) \big| X_{t} =x\Big],
\q\mbox{with}\q A:=[0,c].
\eeaa

\subsubsection*{Dual Bound}

We are interested in deriving an efficient upper bound for $u^{\mathrm{HJB}%
}(0,X_{0}).$ Denoting $R^a_{t}=e^{\int_{0}^{t} a_{s}ds}$, our dual expression is 
\begin{eqnarray}
u^{\mathrm{HJB}}(0,X_{0}) &=&\lim_{h\rightarrow 0}\inf_{\varphi \in 
\mathcal{U}}{\mathbb{E}}\left[ \sup_{a \in \cD_h}\{R^a_T g(X_{T})-\int_{0}^{T}R^a_t\varphi
(t,X_{t})dX_{t}\}\right]  \notag \\
&\leq &\lim_{h\rightarrow 0}{\mathbb{E}}\left[ \sup_{a \in \cD_h}\{R^a_T g(X_{T})-\int_{0}^{T} R^a_t \varphi ^{\ast }(t,X_{t})dX_{t}\}\right] ,  \notag
\end{eqnarray}%
where $\varphi ^{\ast }$ is a fixed strategy. Rewriting the integral in
Stratonovich form, we have%
\begin{eqnarray*}
&&\int_{0}^{T} R^a_t \varphi ^{\ast }(t,X_{t})dX_{t} \\
&=&\int_{0}^{T} R^a_t \varphi ^{\ast }(t,X_{t})\circ dX_{t}-\frac{1%
}{2}\int_{0}^{T} R^a_t \partial_x \varphi
^{\ast } (t,X_{t})\sigma ^{2}(t,X_{t})dt
\end{eqnarray*}%
\noindent Therefore, using the classical Zakai approximation of the
Stratonovich integral, it follows that 
\begin{eqnarray*}
&&{\mathbb{E}}\left[ \sup_{a \in \cD_h}\Big\{ R^a_T g(X_{T})-\int_{0}^{T} R^a_t \varphi ^{\ast }(t,X_{t})dX_{t}\Big\}%
\right] \\
&=&\lim_{n\rightarrow \infty } {\mathbb{E}}\left[ \sup_{a \in \cD_h} \Big\{R^a_T g(X_{T}^{n}) -\int_{0}^{T} R^a_t \varphi ^{\ast }(t,X_{t}^{n})\circ
dX_{t}^{n}+\frac{1}{2}\int_{0}^{T} R^a_t \partial_x \varphi ^{\ast }(t,X_{t}^{n})\sigma ^{2}(t,X_{t}^{n})dt \Big\}
\right] \\
&=&\lim_{n\rightarrow \infty} \mathbb{E}\left[ \sup_{a \in \cD_h} \Big\{R^a_T g(X_{T}^{n}) -\int_{0}^{T}R^a_t\Big(\varphi ^{\ast
}(t,X_{t}^{n})\sigma (t,X_{t}^{n})\dot{W}_{t}^{n}-\frac{1}{2}\partial_x\varphi^{\ast} ( t,X_{t}^{n})\sigma
^{2}(t,X_{t}^{n})\Big)dt\right] \\
&\leq &\lim_{n\rightarrow \infty }\mathbb{E}\left[ \sup_{a \in 
\mathcal{\tilde{D}}} \Big\{R^a_T g(X_{T}^{n}) -\int_{0}^{T} R^a_t \Big(\varphi ^{\ast
}(t,X_{t}^{n})\sigma (t,X_{t}^{n})\dot{W}_{t}^{n}-\frac{1}{2}\partial_x\varphi ^{\ast}(t,X_{t}^{n})\sigma
^{2}(t,X_{t}^{n})\Big)dt\right] ,
\end{eqnarray*}%
where $\mathcal{\tilde{D}}:= \big\{a:[0,T]\rightarrow \dbR \big| ~\mbox{a is measurable, and}~ 0\le a_t \le c~\mbox{for all $t\in [0,T]$} \big\}$. For almost every $\omega $ we may consider for all $n$ the
following deterministic optimization problem. Set 
\begin{eqnarray*}
&g_{\omega ,n}=g(X_{T}^{n}(\omega )),\quad \alpha _{\omega ,n}(t)=-\varphi
^{\ast }(t,X_{t}^{n}\left( \omega \right) )\sigma (t,X_{t}^{n}(\omega ))\dot{%
W}_{t}^{n}\left( \omega \right) ,& \\
&\beta _{\omega ,n}\left( t\right) =\frac{1}{2}
\partial_x \varphi ^{\ast } (t,X_{t}^{n}\left( \omega \right) )\sigma
^{2}(t,X_{t}^{n}(\omega )),&
\end{eqnarray*}%
and consider the function: 
\begin{eqnarray*}
u_{\omega ,n}^{\mathrm{HJ}}(t)=\sup_{a \in \tilde\cD}\Big\{%
 \frac{R^a_T}{R^a_t} g_{\omega ,n}+\int_{t}^{T} \frac{R^a_s}{R^a_t}
\big(\alpha _{\omega ,n}(s)+\beta _{\omega ,n}\left( s\right) \big)ds\Big\}.
\end{eqnarray*}%
\noindent Note that $u^{\mathrm{HJ}}$ is the solution of the (path-wise)
Hamilton-Jacobi equation: 
\begin{equation*}
(u_{\omega ,n}^{\mathrm{HJ}})^{\prime }\left( t\right) +c\left( u_{\omega
,n}^{\mathrm{HJ}}\left( t\right) \right) ^{-}+\alpha _{\omega ,n}(t)+\beta
_{\omega ,n}\left( t\right) =0,\quad u_{\omega ,n}^{\mathrm{HJ}%
}(T)=g_{\omega ,n}.
\end{equation*}

\noindent The ODE for $u_{\omega ,n}^{\mathrm{HJ}} $ can be solved
analytically. Fix a $t^{0}\in \lbrack 0,T]$, and let%
\begin{equation*}
t^{\ast }=\sup \left\{ s<t^{0}:u_{\omega ,n}^{\mathrm{HJ}}(t^{0})u_{\omega
,n}^{\mathrm{HJ}}(s)<0\right\} \vee 0.
\end{equation*}%
For all $t\in \left[ t^{\ast },t_{0}\right] $ we get the following
recurrence equation: 
\begin{eqnarray*}
u_{\omega ,n}^{\mathrm{HJ}}(t) &=&\left\{ 
\begin{array}{lll}
\displaystyle-\int_{t}^{t^{0}}e^{-c(s-t)}\big(\alpha _{\omega ,n}(s)+\beta
_{\omega ,n}(s)\big)ds+u_{\omega ,n}^{\mathrm{HJ}}(t^{0})e^{c(t^{0}-t)},%
\quad u_{\omega ,n}^{\mathrm{HJ}}(t^{0})<0 &  &  \\ 
\displaystyle-\int_{t}^{t^{0}}\big(\alpha _{\omega ,n}(s)+\beta _{\omega
,n}(s)\big)ds+u_{\omega ,n}^{\mathrm{HJ}}(t^{0}),\quad \quad \quad \quad
\quad ,\quad u_{\omega ,n}^{\mathrm{HJ}}(t^{0})>0 &  & 
\end{array}%
\right. ,\text{ } \\
u_{\omega ,n}^{\mathrm{HJ}}(T) &=&g_{\omega ,n}.
\end{eqnarray*}

\noindent Finally, we observe that, 
\begin{equation*}
u^{\mathrm{HJB}}(0,X_{0})\leq \lim_{n\rightarrow \infty }{\mathbb{E}}\big[%
u_{\omega ,n}^{\mathrm{HJ}}(0)\big].
\end{equation*}%
We illustrate the quality of our bounds by the following numerical example.

\begin{rem}
This example falls into the framework of \cite{DB} and \cite{DFG}. By virtue of
their (continuous) pathwise analysis the upper bounds derived above could in
the limit be replaced with equalities. Only the error introduced by the
choice of $\varphi ^{\ast }$ remains.
\end{rem}

\subsubsection*{Numerical example}

We take $\sigma (t,x)=1$, $T=1$ year, $X_{0}=0$. $g(x)=x$. We use two
choices: $\varphi ^{\ast }(t,x)=e^{-c(T-t)}$ (which corresponds to $\partial
_{x}u^{\mathrm{HJB}}$ at the first-order near $c=0$) and $\varphi ^{\ast
}(t,x)=0$. We have computed ${\mathbb{E}}\big[u_{\omega ,n}^{\mathrm{HJ}}(0)%
\big]$ as a function of the time discretization (see Table \ref{tab1} and %
\ref{tab2}). The exact value has been computed using a one-dimensional PDE
solver (see column PDE). We have used different values of $c$ corresponding
to a probability of default at $T$ equal to $(1-e^{-cT})$.

\begin{table}[b]
\centering
\begin{tabular}{|c|c|c|c|c|c|c|c|c|c|}
\hline
$c\;,(1-e^{-cT})$ & PDE & 1/2 & 1/4 & 1/8 & 1/12 & 1/24 & 1/50 & 1/100 & 
1/200 \\ \hline
$0.01\;(1\%)$ & 0.26 & 0.23 & 0.25 & {0.26} & 0.26 & 0.26 & 0.26 & 0.26 & 
\textbf{0.26} \\ \hline
$0.05\;(4.9\%)$ & 1.29 & 1.14 & 1.22 & 1.26 & 1.27 & {\ 1.28} & {\ 1.29} & {%
\ 1.29} & \textbf{1.29} \\ \hline
$0.1\;(9.5\%)$ & 2.52 & 2.24 & 2.39 & 2.46 & 2.48 & 2.51 & 2.52 & {2.52} & 
\textbf{2.52} \\ \hline
$0.7\;(50.3\%)$ & 13.60(0) & 12.63(1) & 13.25(2) & 13.53(5) & 13.61(7) & 13.71(18) & 13.75(44) & 
13.77(112)  & \textbf{13.77} \\ \hline
\end{tabular}%
\caption{The numerical results of ${\mathbb{E}}\big[u_{\protect\omega ,n}^{\mathrm{HJ}}(0)\big]$ with the different time steps when $\protect\varphi^{\ast
}(t,x)=e^{-c(T-t)}$. The numbers in the brackets indicate the CPU times (Intel Core 2.60GHz)  in  seconds for the case $c=0.7$ with $N=8192$ Monte-Carlo paths. }
\label{tab1}
\end{table}

\begin{table}[tbp]
\centering
\begin{tabular}{|c|c|c|}
\hline
$c\;,(1-e^{-cT})$ & PDE & $\dbE\big[u_{\omega ,n}^{\mathrm{HJ}}(0)\big]$ \\ 
\hline
$0.01\;(1\%)$ & 0.26 & 0.40 \\ \hline
$0.05\;(4.9\%)$ & 1.30 & 1.95 \\ \hline
$0.1\;(9.5\%)$ & 2.53 & 3.80 \\ \hline
$0.7\;(50.3\%)$ & 13.60 & 20.08 \\ \hline
\end{tabular}%
\caption{The numerical results of ${\mathbb{E}}\big[u_{\protect\omega ,n}^{\mathrm{HJ}}(0)\big]$ when 
$\protect\varphi ^{\ast }(t,x)=0$. }
\label{tab2}
\end{table}

The approximation has two separate sources of error. First, there is the
suboptimal choice of the minimizer $\varphi ^{\ast }$ for the discretized
optimization implying an upper bias. The second error arises from the
discretization of the deterministic optimization problems, which could underestimates the true value of the optimization. The choice $%
\varphi ^{\ast }=e^{-c(T-t)}$ in our example - as expected - is close to be
optimal, so the errors arising from the discretization dominate. In the contrary, the choice $\f^* =0$ is far from being optimal, so the numerical results are much bigger than the value function.

\end{document}